\documentclass{amsart}
\usepackage{amssymb,amsmath, amsthm,latexsym}
\usepackage{graphics}
\usepackage{psfrag}
\usepackage{tikz-cd}
\usepackage{amscd}
\usepackage{graphicx}
\newcommand{\cal}[1]{\mathcal{#1}}
\theoremstyle{plain}
\newtheorem{theo}{Theorem}

\newtheorem*{question}{Question}
\newtheorem{lemma}{Lemma}[section]
\newtheorem{theorem}[lemma]{Theorem}
\newtheorem{proposition}[lemma]{Proposition}
\newtheorem{corollary}[lemma]{Corollary}
\theoremstyle{definition}

\newtheorem{definition}[lemma]{Definition}
\newtheorem{remark}[lemma]{Remark}
\newtheorem{example}[lemma]{Example}
\let\egthree=\phi
\let\phi=\varphi
\let\varphi=\egthree

\begin{document}
\title{On the cohomology of strata of abelian differentials}
\author{Ursula Hamenst\"adt}
\thanks{Partially supported by the Hausdorff Center Bonn}
\date{November 11, 2020}


\begin{abstract}
For $g\geq 3$, we study the cohomology classes in 
the closure of a stratum of abelian differentials defined by
the boundary strata of codimension one. 
As an application, we find an explicit stratification 
of the spin moduli space for an odd spin structure consisting
of $g-1$ strata ${\cal D}_j$ of codimension $j-1$ such that
${\cal D}_j$ does not contain a complete variety for all $j$.
We also recover some results of Korotkin and Zograf and 
of Chen using a unified 
topological argument. 
\end{abstract}

\maketitle

\section{Introduction}

For $g\geq 3$ the 
 \emph{moduli space} ${\cal M}_g$ 
of complex curves of genus $g$ is a complex orbifold. More precisely, 
it is the quotient of a bounded domain in $\mathbb{C}^{3g-3}$, the so-called 
\emph{Teichm\"uller space} ${\cal T}_g$ of genus $g$, 
under the action of a discrete group of biholomorphic automorphisms, 
the \emph{mapping class group} ${\rm Mod}(S_g)$. 
The following question can be found in \cite{FL08}, see also \cite{HL98} for a motivation.

\begin{question} 
Does ${\cal M}_g$ admit a stratification with all strata affine subvarieties
of codimension $\leq g-1$?
\end{question} 

The moduli space admits a compactification $\overline{\cal M}_g$, the 
so-called \emph{Deligne Mumford compactification}, which equips 
${\cal M}_g$ with the structure of a quasi-projective variety. 
The complement of 
an irreducible effective ample divisor in 
$\overline{\cal M}_g$ is affine. This was used by
Fontanari and Looijenga to show that the complement of the 
\emph{Thetanull} divisor
in ${\cal M}_g$ parameterizing curves with an effective even theta characteristic 
is affine for every $g\geq 4$ (Proposition 2.1 of \cite{FL08}). They also show
that the answer to the question is yes for all $g\leq 5$. 
Another approach towards an answer to this question which is 
closer to our viewpoint is due to Chen \cite{Ch19}.

The main goal of this article is to give some additional evidence that the answer
to the above question is affirmative. To this end 
consider the \emph{Hodge bundle} over ${\cal M}_g$ whose
fiber over a complex curve $X$ is just the $g$-dimensional
vector space of holomorphic one-forms on $X$.
The projectivization $P:{\cal P}\to {\cal M}_g$ 
of the Hodge bundle if a holomorphic fiber bundle over ${\cal M}_g$ in the orbifold
sense. It 
admits a natural stratification whose strata consist of projective differentials with the 
same number and multiplicities of zeros.
These strata need not be connected,
but the number of connected components is at most 3 \cite{KtZ03}.

The \emph{tautological ring} of ${\cal M}_g$ is the subring of the rational
cohomology ring of ${\cal M}_g$ generated by the 
\emph{Mumford Morita Miller} classes $\kappa_k\in H^{2k}({\cal M}_g,\mathbb{Q})$
(see \cite{M87} and \cite{Lo95} for a comprehensive discussion of these classes).  
Denote by $\eta$ the 
Chern class of the tautological
line bundle over the fibers of ${\cal P}$.
We use the zeros of the differentials in a component ${\cal Q}$ of 
a stratum to analyze the cohomology classes on the closure
$\overline{\cal Q}$ of ${\cal Q}$ 
defined by the boundary components
of ${\cal Q}$ of codimension one. 
We find that these classes are all contained in the subspace of 
$H^*(\overline{\cal Q},\mathbb{Q})$ spanned by the restrictions to 
$\overline{\cal Q}$ of the pull-back
$P^*\kappa_1$ and the class 
$\eta$.


Denote by 
$\mathbb{P}{\cal H}(k_1,\dots,k_m)$ the stratum of projective
abelian differentials with $m$ zeros of order $k_j$ (here
$k_j\geq 1$ and 
$\sum_jk_j=2g-2$). As an application, we obtain a topological proof of 
the following result of Chen \cite{Ch17}. 

\begin{theo}\label{strata}
  Let ${\cal Q}\subset \mathbb{P}{\cal H}(k_1,\dots,k_m)$
  be a component of a stratum of projective abelian
differentials with $m\geq 1$ zeros of order $k_i$ $(i\leq m)$;
then for all $\ell \geq 1$ we have 
\[P^*\kappa_\ell\vert {\cal Q}= (-1)^{\ell +1}
\sum_i \bigl( \sum_{j=0}^{\ell} 
\frac{k_i}{(k_i+1)^{\ell-j}} \bigr)\eta^\ell.\]
In the case $\ell=1$ this reads
\[P^*\kappa_1\vert {\cal Q}=\sum_i(k_i+1-\frac{1}{k_i+1})\eta\vert {\cal Q}.\]
In particular, the restriction to ${\cal Q}$ of the pull-back of the
tautological ring of ${\cal M}_g$ coincides with the subring 
of $H^*({\cal Q},\mathbb{Q})$ generated by
the restriction of $\eta$.
\end{theo}


Since all but the first Mumford Morita Miller classes vanish 
on ${\cal M}_3$ \cite{Lo95}, Theorem \ref{strata} 
for two strata in $g=3$ is also due to Looijenga and Mondello \cite{LM14}.

The \emph{moduli space of curves with odd theta characteristic}
${\cal M}_{g,{\rm odd}}$ is the moduli space of pairs 
$(X,L)$ where $X$ is a complex curve of genus $g$ and where $L$ is 
a square root of the canonical bundle so that $h^0(X,L)$ is odd. 
This is a finite orbifold cover of ${\cal M}_g$. We use our cohomological 
computation to give some evidence towards the
question in \cite{FL08}. We show

\begin{theo}\label{stratificationthm}
The spin moduli space ${\cal M}_{g,{\rm odd}}$ admits 
an explicit stratification
into complex strata ${\cal D}_j$ of codimension $j-1$ 
$(1\leq j\leq g-1)$ such that for all $j\leq g-1$, 
the restriction of the class $\kappa_1$ to 
the stratum 
${\cal D}_j$ vanishes. In particular, ${\cal D}_j$ 
does not contain a complete
subvariety. 
\end{theo}

The stratum ${\cal D}_j$ is defined as follows. Let 
${\cal Q}=\mathbb{P}{\cal H}(2,\dots,2)^{\rm odd}$
be the component of the stratum of abelian differentials with all zeros of order 2 and
odd parity \cite{KtZ03}. The closure $\overline{\cal Q}$ of ${\cal Q}$ 
in ${\cal P}$ projects onto ${\cal M}_{g,{\rm odd}}$. 
For $j\leq g-1$ let $\overline{{\cal Q}}_j\subset
\overline{\cal Q}$ be the closure of the union of all
boundary components of ${\cal Q}$ of codimension $j-1$ and define
${\cal D}_j=P\overline{\cal Q}_j-P\overline{\cal Q}_{j+1}$. Note that ${\cal D}_{g-1}$
is the projection of the union of those components of 
$\mathbb{P}{\cal H}(2g-2)$ which have an odd spin structure. 
The number of such components is one for $g\equiv 0,3$ mod $4$, and it
equals two otherwise. We conjecture that the 
strata ${\cal D}_j$ are in fact affine for all $j$.

That components of strata do not contain complete subvarieties is due
to Gendron \cite{G20}.

The organization of this article is as follows. 
In Section \ref{oversurfaces} we study the pull-back $P^*{\cal C}$ of the 
\emph{universal curve} ${\cal C}\to {\cal M}_g$ 
to the moduli space of projective abelian differentials. 
We obtain some information on the 
cohomology class $P^*\kappa_1$ by analyzing 
the subvariety of $P^*{\cal C}$ which intersects the fiber over $q$ in the zeros of $q$.
This locus can be used to gain some information on $P^*\kappa_1$ via 
Poincar\*e duality in 
surface bundles as in \cite{H20}.

In Section \ref{thetauto} we begin the study of the second
cohomology group of closures of components of strata, and we 
establish Theorem \ref{strata}. 
This is used in
Section \ref{signatureasintersection} to give a purely topological
proof of the following result of Korotkin and Zograf.
For its formulation, recall that the rational 
cohomology ring of the projectivized Hodge bundle
$P:{\cal P}\to {\cal M}_g$ is the ring 
\[H^*{\cal P},\mathbb{Q})=P^*H^*({\cal M}_g,\mathbb{Q})[\eta]/
\bigl(\eta^{g}+c_1({\cal H})\eta^{g-1}+\cdots +c_g({\cal H})\bigr)\]
where $\eta$ is the tautological class of the fiber and $c_i({\cal H})$ is the
$i$-th Chern class of ${\cal H}$. These Chern classes are polynomials in the 
odd Mumford Morita Miller classes (see Section 2 of \cite{M87}). 
Let $\xi\in H^2({\cal P},\mathbb{Q})$
be the class dual to the stratum ${\cal P}(1)$ of codimension one.

\begin{theo}[Korotkin and Zograf \cite{KZ11}]\label{kozo}
\[\xi=2P^*\kappa_1-(6g-6)\eta.\]
\end{theo}

The computation in \cite{KZ11} extends to the boundary of the 
Deligne Mumford compactification, however we do not pursue 
such a computation in this work. An algebraic geometric proof is due to Chen \cite{Ch13}.

In Section \ref{boundary} we obtain some information on the second cohomology classes 
of the closure of a stratum defined by its codimension one  
boundary strata. 
This is then used in Section \ref{stratification}
to show Theorem \ref{stratificationthm}. 


\bigskip
\noindent
{\bf Acknowledgement:} I am indebted to Dawei Chen for pointing
out an error in a computation in an earlier
version of this paper which made some
part of the results invalid. I am also indebted to him for
showing me the reference \cite{Ch17} which contains an earlier 
elegant proof of Theorem \ref{strata}.  
I am moreover grateful to Dawei Chen and Samuel Grushevsky for
useful discussions.

\section{The zero sets of strata of abelian differentials}\label{oversurfaces}

The goal of this section is to establish some geometric properties of strata of abelian differentials
and use this to obtain some first information on the first Mumford Morita Miller class. 
Throughout we assume that $g\geq 3$.

Let
$\Upsilon:{\cal C}\to {\cal M}_g$
be the \emph{universal curve}, that is, the fiber bundle
(in the orbifold sense) whose fiber over a point $X\in {\cal M}_g$
is just the Riemann surface $X$.
Consider the pull-back 
\[\Pi:P^*{\cal C}\to {\cal P}\]
of the universal curve to the projectivized Hodge bundle.
For each $q\in {\cal P}$, the zeros of $q$ define a subset of
the fiber of $P^*{\cal C}$ of cardinality at most $2g-2$.
Denote by $\Delta\subset P^*{\cal C}$ the locus of all these
zeros.

%

The following is Proposition 2.2 of \cite{H20}.
For its formulation, a closed subvariety $Y$ of codimension one of
a smooth variety $X$ is a \emph{local complete intersection}
if the ideal sheaf ${\cal F}_Y$ of $Y$ in $X$ can be locally generated
by a single element at every point.  

\begin{proposition}\label{variety}
  The subset $\Delta\subset P^*{\cal C}$
  is a subvariety of $P^*{\cal C}$  of codimension one, and
  it is a local complete intersection.
\end{proposition}

Write ${\cal P}=\cup_k{\cal P}(k)$ where for $0\leq k\leq 2g-3$ the set 
${\cal P}(k)$ is the
locus of projective differentials with precisely $2g-2-k$ zeros.
Then ${\cal P}(k)$ is a smooth suborbifold of ${\cal P}$ of codimenision
$k$, and it is a disjoint union of strata.
Furthermore, we have $\overline{\cal P}(k)=\cup_{j\geq k}{\cal P}(j)$
and hence this decomposition gives ${\cal P}$ the structure of a
complex stratified space. The set
${\cal P}(1)$ consists of the single stratum
$\mathbb{P}{\cal H}(1,\dots,1,2)$.

Now let us consider a component ${\cal D}\subset{\cal P}$
of a stratum of projective abelian
differentials. By Proposition \ref{variety},
if $q\in {\cal D}$ and if $z\in P^*{\cal C}$ is a point in the
fiber of $P^*{\cal C}$ over $q$ which is a zero of $q$ of order $k\geq 1$, then
there is a neighborhood $U$ of $q$ in ${\cal D}$ and a holomorphic
section $\zeta:U\to P^*{\cal C}$ with image in $\Delta$ and such
that $\zeta(q)=z$. If $z$ is the only zero of $q$ of order $k$, then the map which associates
to a differential in ${\cal D}$ its unique zero of order $k$ is a global holomorphic section of 
$P^*{\cal C}$ over ${\cal D}$.

The following proposition gives some further information on
the variety $\Delta$.
It is a more explicit version of a result in Section 8 of \cite{EMZ03}.

\begin{proposition}\label{addboundarycomponent}
  Let ${\cal Q}\subset {\cal P}$ be a component of a
  stratum of projective abelian
differentials and let ${\cal D}$ be an irreducible 
component of codimension one of the boundary of 
${\cal Q}$,  
obtained by
colliding two zeros of differentials in ${\cal Q}$ of order $m_1,m_2\geq 1$ to a single zero
of order $m=m_1+m_2\geq 2$.
\begin{enumerate}
\item  
Assume that the differentials in ${\cal D}$ have 
a single zero of order $m$.
Then 
${\cal Q}\cup {\cal D}\subset {\cal P}$ is a smooth complex orbifold. 
 Let $\zeta:{\cal D}\to P^*{\cal C}$ be 
the holomorphic section defined by the zero of order $m$. 
\begin{itemize}
\item
If $m_1=m_2$ then 
the normal bundle of ${\cal D}\subset {\cal Q}\cup {\cal D}$ is isomorphic to the 
square $\zeta^*(\nu)^2$ of the pull-back $\zeta^*(\nu)$ of the vertical 
tangent bundle of $P^*{\cal C}$ along 
$\zeta$. 
\item
If $m_1\not=m_2$ 
then the normal bundle of
${\cal D}\subset {\cal Q}\cup {\cal D}$ is isomorphic to 
$\zeta^*(\nu)$. 
\end{itemize}
\item If ${\cal D}$ consists of differentials with $k\geq 2$ zeros of order $m$,
then ${\cal Q}\cup {\cal D}$ has a normal crossing singularity along ${\cal D}$
consisting of $k$ smooth local branches which intersect transversely along ${\cal D}$.
\end{enumerate}
\end{proposition}
\begin{proof} Consider first the case that differentials in ${\cal D}$ have a unique
zero of order $m$. Equivalently, 
a zero of a differential $q\in {\cal D}$ arising from a collision of two 
zeros of a differential in ${\cal Q}$ 
is distinguished by its multiplicity. 

Consider the preimages ${\cal D}_0$ and $ {\cal Q}_0$ of ${\cal D}$ and 
${\cal Q}$, respectively, in the moduli space of abelian differentials, that is, in 
the complement
${\cal H}^*$ of the zero section of the Hodge bundle ${\cal H}$.
Then ${\cal D}_0,{\cal Q}_0$ admit a natural holomorphic action of the group 
$\mathbb{C}^*$ by complex
multiplication, with quotients ${\cal D}$ and ${\cal Q}$. 

A neighborhood  
of $q\in {\cal D}_0$ in ${\cal Q}_0\cup {\cal D}_0$ 
is obtained from a neighborhood of $q$ in ${\cal D}_0$ 
by opening the distinguished zero of order $m\geq 2$ 
 to two zeros of prescribed order $m_1\leq m_2$ with $m_1+m_2=m$  
as explained in Section 8 of \cite{EMZ03}. We have to show that this
operation is equivariant with respect to the $\mathbb{C}^*$-action and 
compatible with the complex structure on the quotients ${\cal D},{\cal Q}$, 
and we have to compute the 
normal bundle. 

We proceed as on p.86 of \cite{EMZ03}. An 
abelian differential $q$ on 
a Riemann surface $X$ of genus $g$  
determines a flat metric on $X$ 
in the conformal class of $X$,
with singularities at the zeros of $q$.
Assume that $q\in {\cal D}_0$, let 
$x\in X$ be the distinguished zero of $q$ and let $\epsilon >0$ be sufficiently
small that the closed disk $D(\epsilon)$ of radius
$\epsilon$ about $x$ for the flat 
metric defined by $q$ is a topological disk embedded in $X$.
Let $\delta <\epsilon/2$ and let $\gamma$ be a straight line segment of length
$\delta$ with one endpoint at $x$, parameterized 
proportional to arc length on the interval $[0,1]$. 
We claim that $\gamma$ determines uniquely
a point in ${\cal Q}_0$ with a saddle connection of
length $2\delta$ connecting a zero
of order $m_1$ to a zero of order $m_2$. 

Namely, the closed disk $D(\epsilon)$ can be represented as a union of $2m+2$ 
flat half-disks of radius $\epsilon$.  Their oriented straight line boundary segments are 
segments whose direction for the flat metric is up to sign the direction of 
$\gamma$. These half-disks are 
glued in circular order along the half-segments of length $\epsilon$. 
Let $\hat \gamma$ be the straight line segment 
of the same length $\delta$ as $\gamma$, with one endpoint at $x$,
which makes an 
angle of $(2m_1+1)\pi$ with $\gamma$, measured for the flat metric in counter 
clockwise direction. The direction of $\hat \gamma$ for the flat metric 
is opposite to the direction of $\gamma$. 
Both oriented segments $\gamma,\hat \gamma$
are contained in the oriented boundary of one of the 
embedded flat half-disk of radius $\epsilon$ determined by the direction of 
$\gamma$, with center at $x$. 

Cut $D(\epsilon)$ open along the line segments in the boundary of these 
two half-disks containing $\gamma,\hat \gamma$ and 
glue these two half-disks along the segment of length $2\delta$
centered at $x$, leaving a pair of free line segments of  length $\epsilon-\delta$ on 
the boundary of each 
of the half-disks. The remaining
half-disks determined by the direction of $\gamma$
can be glued to these two half-disks isometrically
along the boundary in a circular
fashion as illustrated on p.87 of \cite{EMZ03}.
The result of this construction is
a new  flat metric on the surface $S_g$, of the same area. 
This flat metric is defined by an abelian 
differential $q(\gamma)\in {\cal Q}_0$ which is uniquely
determined by $q$, the choice of $\gamma$
and the decomposition $m=m_1+m_2$. 
It has a distinguished saddle connection of length $2\delta$ connecting
the two newborn zeros of order $m_1,m_2$. If we denote by
$x(\gamma)$ the midpoint of this saddle connection, then the
complements of the disks of radius $\epsilon$ about $x$
and $x(\gamma)$ for the flat metrics defined by $q,q(\gamma)$ are 
isometric. 

For fixed $q\in {\cal D}_0$ and 
as the endpoint of the geodesic 
segment $\gamma$ different from $x$ 
varies in the punctured disk of radius
$\epsilon/2$ about $x$, the above
construction defines a 
family of abelian differentials in ${\cal Q}_0$
depending on a complex parameter varying in a punctured 
disk in $\mathbb{C}$, that is,
in a punctured coordinate disk about $x$ in $X$.
As explained on p.87 of \cite{EMZ03}, for a suitable choice of a basis of 
relative homology of the closed 
surface $S_g$ of genus $g$, marked at the zeros of $q$, 
all but perhaps one coefficient of the corresponding
period coordinates are constant, and the remaining period coordinate 
is changed by $-\gamma$ (the missing factor $1/2$ in our description 
stems from a slight variation in the setup). 

As a consequence, this construction gives rise to a holomorphic map from 
of a disk in $\mathbb{C}$ into ${\cal H}^*$ which intersects ${\cal D}_0$ in the single point $q$. 
As it commutes with 
multiplication of a differential with a nonzero complex number, it descends to 
a holomorphic map from of a disk in $\mathbb{C}$ into ${\cal P}$ which intersects ${\cal D}$   
in a single point, and this point is the projection of $q$. Furthermore, by naturality 
with respect to suitable period coordinates, chosen as in the previous
paragraph, it depends in a holomorphic 
fashion on $q\in {\cal D}_0$. 

If $m_1=m_2=m/2$, then the two flat surfaces obtained from this construction from segments
$\gamma_1,\gamma_2$ of the same length $\delta$ 
which make an angle of $(m+1)\pi$ at $x$ for the flat
cone metric, are isometric. As a consequence,  
the holomorphic involution $z\to -z$ in the tangent
space of $X$ at the distinguished zero of order $m$
extends to an involution
of the local parameter space for opening a zero of order $m$,
and the map which associates to 
a point in this local parameter space the resulting
area one abelian differential factors through the quotient
of this involution. 

Lemma 8.1 of \cite{EMZ03} shows that in the open and dense subset
of ${\cal D}_0$ consisting of flat metrics which do not
admit any isometry, this is the only 
identification. More precisely,  the lemma states that
each direction for the flat metric of $q$ gives rise to precisely $m+1$ distinct flat surfaces with
\emph{labeled} zeros.
The isometry between two of these flat metrics arising from the involution $z\to -z$ 
as discussed in the previous paragraph exchanges the two newborn zeros of order $m_1=m_2$ 
of the differentials in ${\cal Q}_0$
and hence changes the labels. As we do not fix labels here, by
Lemma 8.1 of \cite{EMZ03}   
the above construction defines a holomorphic 
parameterization of a neighborhood of
${\cal D}_0$ in ${\cal Q}_0\cup {\cal D}_0$. By equivariance under the
action of $\mathbb{C}^*$, this parameterization descends to a
parameterization of a neighborhood of ${\cal D}$ in ${\cal Q}$.
Moreover, 
the normal bundle of ${\cal D}$ in ${\cal Q}$ is the square of the 
pull-back of the vertical tangent bundle of 
$P^*{\cal C}$ at the distinguished zero
as described in the proposition as its fiber over a fixed zero $x$ is doubly covered by the
fiber of the vertical tangent bundle at $x$.

If $m_1\not=m_2$, then locally near ${\cal D}_0$ the two newborn zeros of the
differentials arising from
the above construction can not be exchanged. Thus in this case Lemma 8.1 of \cite{EMZ03} 
shows that the above construction defines a 
holomorphic parameterization of a neighborhood of ${\cal D}_0$ in 
${\cal Q}_0$ and hence it parameterizes a neighborhood of ${\cal D}$ in ${\cal Q}$. 
Furthermore, the normal bundle of ${\cal D}$ equals the pull-back of the 
vertical tangent bundle of $P^*{\cal C}$ at the distinguished zero of order $m$. 
This shows  the first part of the proposition.

Now let us assume that differentials in ${\cal D}$ have $k\geq 2$ zeros of order
$m$. Let $q\in {\cal D}$ and let 
$q_0\in {\cal D}_0$ be a preimage of $q$. 
If $x_1\not=x_2$ are two zeros of the same order $m$ for $q_0$,
then for differentials in a 
neighborhood $V_0$ of $q_0$ in ${\cal D}_0$ we can open up 
the zero $x_1$ locally in a neighborhood of $x_1$, preserving the flat metric
on the complement of a small disk about $x_1$, in particular near $x_2$, 
and we obtain a differential in ${\cal Q}_0$. This construction
determines the structure of a smooth complex orbifold on the union of the projection
$V$ of $V_0$ to ${\cal D}$ with
some open subset $U(x_1,V)$ of ${\cal Q}$ which is compatible with the complex
structure and the topology of ${\cal P}$.
Similarly, 
preserving $x_1$ and opening $x_2$  gives rise to the structure of a smooth complex
orbifold on the union of $V$ with an open subset $U(x_2,V)$ of ${\cal Q}$. 

If the flat metric 
defined by $q\in {\cal D}_0$ does not admit an isometry which exchanges $x_1$ and $x_2$, that is, if
$q$ belongs to the open and dense set of smooth points of the orbifold 
${\cal P}$, then for a suitable choice of the neighborhood $V_0$ of $q$, the sets
$U(x_1,V)$ and $U(x_2,V)$ are disjoint from each other, and the unions
$U(x_1,V)\cup V$ and $U(x_2,V)\cup V$ intersect transversely along $V$.
As this construction is local, it can be extended to more than two zeros of the same order $m$
and yields 
the second part of the proposition.
\end{proof}

By Proposition \ref{variety}, the zeros of 
a projective abelian differential $q$ on a Riemann surface $X$ define a
codimension one complex subvariety $\Delta$ in the holomorphic fiber bundle
$P^*{\cal C}\to {\cal P}$.
Over a fixed stratum 
${\cal Q}$, this subvariety is just a holomorphic
\emph{multisection}, that is,
it can locally be described as consisting of $\ell$ holomorphic
sections of $P^*{\cal C}\to {\cal P}$ where $\ell\geq 1$ is the
number of zeros of differentials in ${\cal Q}$. Note that this a purely local
statement. 

Our next goal is to describe the behavior of these multisections
as the differentials approach
a boundary component of ${\cal Q}$ of codimension one,
given by a collision of two of
these zeros.

For the formulation of this description, recall that a two-sheeted holomorphic 
branched covering $\zeta:S\to B$ of two complex
curves $S,B$, branched at a point $x\in S$, is given in suitable
holomorphic coordinates $z,w$ on $S,B$ 
 near $x$ and $\zeta(x)$, respectively, with $x=\{z=0\},\zeta(x)=\{w=0\}$,
 as $w=z^2$. The following definition is a special case of well known constructions
 and is included here for clarity of the exposition.

\begin{definition}\label{branchedcover} 
For a number $m\geq 2$, an \emph{$m$-sheeted holomorphic branched covering}
of two complex manifolds $M, N$ of the same complex dimension, 
\emph{doubly branched} along a complex hypersurface $H\subset M$, is a surjective holomorphic map
$\zeta:M\to N$ with the following properties.
\begin{enumerate}
\item The restriction of $\zeta$ to $H$ is a
biholomorphism onto its image. 
\item The restriction of $\zeta$ to $M-\zeta^{-1}(\zeta(H))$ is an $m$-sheeted unbranched holomorphic
covering.
\item There exists a neighborhood $V$ of $\zeta^{-1}(\zeta(H))-H$ such that the restriction of 
$\zeta$ to $V$ is an $(m-2)$-sheeted unbranched holomorphic covering. 
\item For every $x\in H$ there are holomorphic coordinates
$(z_1,\dots,z_n)$ on an open neighborhood $U$ of $x$ in 
$M$ and holomorphic coordinates $(w_1,\dots,w_n)$ on a neighborhood of 
$\zeta(x)$ in $N$,
with $H\cap U=\{z_1=0\}$ and the property that in these coordinates, the map $\zeta$ is defined
by $(z_1,z_2,\dots,z_n)\to (z_1^2,z_2,\dots,z_n)$.
\end{enumerate}
\end{definition}

We also have to look at a two-sheeted
holomorphic branched covering from a singular variety $M$ 
onto a smooth complex manifold $N$
which is branched along a normal crossing divisor of $M$
in the following sense.

\begin{definition}\label{branchedcover2}
Let $Z$ be a codimension one 
complex subvariety of a smooth complex variety $M$, 
smooth away from a codimension one subvariety
$H\subset Z$, and assume that $Z$ has a normal crossing 
singularity along $H$.
A \emph{holomorphic branched covering} of $Z$ onto
a smooth complex variety $N$, doubly branched along the singular hypersurface $H$, 
is a surjective holomorphic map $\zeta:Z\to N$ with the following properties.
\begin{enumerate}
\item The restriction of $\zeta$ to $H$ is a
biholomorphic map onto its image. 
\item The restriction of $\zeta$ to $Z-\zeta^{-1}(\zeta(H))$ is 
an $m$-sheeted unbranched holomorphic
covering.
\item There exists a neighborhood $V$ of $\zeta^{-1}(\zeta(H))-H$ such that the restriction 
of $\zeta$ to $V$ is an $(m-2)$-sheeted unbranched holomorphic covering. 
\item For every $x\in H$ there are holomorphic coordinates
$(z_0,z_1,\dots,z_n)$ on a neighborhood $U$ of $x$ in $M$, and holomorphic 
coordinates $(w_1,\dots,w_n)$ on $N$ near $\zeta(x)$,
with $Z\cap U=\{z_0^2-z_1^2=0\}$ and 
$H\cap U=\{z_0=z_1=0\}$ and the property that in these coordinates, the map $\zeta$ is the restriction 
to $Z\cap U$ of the map defined
by $(z_0,z_1,\dots,z_n)\to (z_1,z_2,\dots,z_n)$.
\end{enumerate}
\end{definition}

The following statement uses Definition \ref{branchedcover}
and Definition \ref{branchedcover2} in the orbifold sense. That is,
the definitions (which are mainly local) apply after perhaps 
passing to a finite manifold cover.

\begin{proposition}\label{boundarycomp2}
Let ${\cal D}$ be a codimension one irreducible boundary 
component of 
a stratum ${\cal Q}\subset {\cal P}$,  obtained by
colliding two zeros of differentials in ${\cal Q}$ of order $m_1,m_2\geq 1$ to a single zero
of order $m=m_1+m_2\geq 2$, and let $Z=\Delta\cap 
\Pi^{-1}({\cal Q}\cup {\cal D})$.
\begin{enumerate}
\item
Assume that ${\cal D}$ consists of differentials
with a single zero of order $m$.
\begin{itemize}
\item If $m_1=m_2$ then $Z$ is a smooth complex suborbifold of
$\Pi^{-1}({\cal Q}\cup {\cal D})$ of codimension one.
The projection $\Pi\vert Z:Z\to 
{\cal Q}\cup {\cal D}$ is a holomorphic branched covering, 
doubly branched along the zeros of order $m$ of the differentials in ${\cal D}$.
\item If $m_1\not=m_2$ then $Z$
 has a normal crossing singularity at the zeros of 
order $m$ of the differentials in ${\cal D}$. The projection $\Pi\vert Z:Z\to 
{\cal Q}\cap {\cal D}$ is a holomorphic  branched 
covering, doubly branched along the 
zeros of order $m$ of the differentials in ${\cal D}$.
\end{itemize}
\item
If ${\cal D}$ consists of differentials with $k\geq 2$ 
zeros of order $m$, then property (1) holds true for each of the $k$ local branches of ${\cal Q}$ which 
intersect transversely along ${\cal D}$.
\end{enumerate}
\end{proposition}
\begin{proof}
As in the proof of Proposition \ref{addboundarycomponent},
denote by ${\cal D}_0,{\cal Q}_0$ the preimage of
${\cal D},{\cal Q}$ in the moduli
space of abelian differentials. 
Let $x$ be a zero of order $m$ for an abelian differential
$q\in {\cal D}_0$
and let $\epsilon >0$ be such that the closed disk of
radius $\epsilon $ about $x$ for the flat metric defined by $q$ 
is isometrically embedded in the Riemann surface $X$ underlying $q$.
Then there is a canonical complex coordinate $z$ for
$X$ near $x$, so that in this coordinate,
the differential $q$ equals the differential $z^mdz$.  

Let $\gamma$ be a straight line segment of length $\delta<\epsilon/2$ for the flat metric 
defined by $q$ issuing from $x$. Opening up the zero of $q$ into two zeros of 
order $m_1,m_2$ along $\gamma$ as described in Section 8 of \cite{EMZ03} and recorded
in the proof of Proposition \ref{addboundarycomponent} defines a differential $q(\gamma)$
with a saddle connection of length $2\delta$. The direction of the saddle connection equals the 
direction of $\gamma$, and the midpoint of the saddle connection is the natural image of
the zero $x$ of $q$. 

Let us consider a model for this situation. 
It is given by a complex coordinate $z$ on the Riemann surface underlying $q(\gamma)$, 
containing $0$ in its range, 
a straight line segment through $0$ for the flat metric
defined by $q(\gamma)$
of length $2\delta$, 
and a zero of order $m_1,m_2$ at the endpoints.
For a suitable choice of such a complex coordinate $z$, the differential  
can be represented as 
\[(z-a)^{m_1}(z+a)^{m_2}dz\]
where $a=a(\gamma)\in \mathbb{C}^*$ can be computed from
the length and direction of the saddle connection. 
Solving $(z-a)^{m_1}(z+a)^{m_2}dz=dw$ near $z=0$ expresses the local coordinate 
$w$ describing the flat metric of the differential $q(\gamma)$, normalized to vanish 
at the point $z=0$, as a polynomial of degree $m+1$ 
in the coordinate $z$,
with coefficients depending holomorphically on the complex variable $a\in \mathbb{C}^*$. 

As in the proof of Proposition \ref{addboundarycomponent}, 
for a fixed basis of relative homology of the surface $X$ marked at the
zeros of $q$, period coordinates for the differentials in
a neighborhood of $q$ in ${\cal D}_0$ 
extend to period coordinates on a neighborhood of $q$ in 
${\cal Q}_0\cup {\cal D}_0$ 
using as an extra parameter the distinguished saddle connection between the 
newborn zeros of length less than $\epsilon$ (here $\epsilon >0$ 
is a constant which depends on the flat metric defined by $q$), 
and these coordinates also 
define holomorphic coordinates on a neighborhood $V$ of
$q$ in ${\cal Q}\cup {\cal D}$ by 
equivariance under the action of $\mathbb{C}^*$.
In other words, the differentials 
resulting from this construction
depend in a holomorphic fashion on period coordinates for $q$ and the endpoint of 
the straight line segment $\gamma$.

Let us assume that the differentials in ${\cal D}$ have a
single zero of order $m$.
By Proposition \ref{addboundarycomponent}, in this case 
${\cal D}_0$ is a smooth suborbifold of ${\cal Q}_0\cup {\cal D}_0$. 
Let $\Pi_0:P_0^*{\cal C}\to {\cal H}^*$ be the pull-back of 
the universal curve to the complement ${\cal H}^*$ of the zero section in 
the Hodge bundle. Considering again a differential $q\in {\cal D}_0$ with
a zero $x$ of order $m$, 
the above discussion shows that there 
are holomorphic local functions 
$(z,v_1,\dots,v_k)$ on a neighborhood $U$ of $x$ in $\Pi_0^{-1}({\cal Q}_0\cup {\cal D}_0)$, 
with $\{v_i={\rm const}\}$ defining the foliation into the fibers of the bundle 
$P_0^*{\cal C}\to {\cal H}^*$, and the following additional properties.  
\begin{enumerate}
\item The functions $v_i$ are pull-backs by $\Pi_0$ of holomorphic local functions $\hat v_i$ on 
${\cal Q}_0\cup {\cal D}_0$. 
\item $\{\hat v_1=0\}=\Pi_0(U)\cap {\cal D}_0\subset {\cal D}_0\cup {\cal Q}_0$, 
and  $(\hat v_2,\dots,\hat v_k)$ are holomorphic coordinates for ${\cal D}_0$ (in the orbifold sense).
\item The restriction of the function $z$
to a fiber of $\Pi_0$ over $\Pi_0(U)$ is a holomorphic coordinate on the fiber.
\item
The zeros of order $m$ for the projective differentials in $\Pi_0(U)\cap {\cal D}_0$, 
viewed as points in the fiber of 
$P_0^*{\cal C}$ over the points in ${\cal D}_0$,
are contained in the domain of the fiber coordinate $z$. 
The abelian differentials $u\in \Pi_0(U)$  
are given in the fiber coordinate $z$  
by $u=(z- \hat v_1(u))^{m_1}(z+\hat v_1(u))^{m_2}dz$.  
\end{enumerate}

By Proposition \ref{addboundarycomponent} and its proof, 
in the case $m_1=m_2$, the differentials parameterized by 
$(\hat v_1,\hat v_2,\dots,\hat v_m)$ and $(-\hat v_1,\hat v_2,\dots,\hat v_m)$ coincide and hence 
$(\hat v_1^2,\hat v_2,\dots,\hat v_m)$ are complex coordinates for ${\cal Q}_0\cup {\cal D}_0$
(in the usual sense which equips the quotient of the unit disk
$\{\vert z\vert <1\}$  by the involution
$z\to -z$ with the structure of a Riemann surface, biholomorphic to the disk). 
Putting $w=v_1^2$, 
the equation for the subvariety
$Z=\Delta\cap \Pi_0^{-1}({\cal D}_0\cup {\cal Q}_0)$ 
of $P_0^*{\cal C}$ near the zero $x$ of order $m$ 
equals $z^2-w=0$. For fixed $\hat v_2,\dots,\hat v_k$ 
this locus is parameterized by 
$z\to (z,z^2)$ in the coordinate functions $(z,w)$. 
Thus $Z$ is a smooth suborbifold of $\Pi_0^{-1}({\cal Q}_0\cup {\cal D}_0)$ 
whose tangent space at the zero $x$ of order $m$
of a differential in ${\cal D}_0$ 
contains the tangent space of the fiber of $\Pi_0$ at $x$.
Furthermore, in these coordinates,
near the zero of order $m$ the projection map $\Pi_0\vert Z:Z\to {\cal Q}_0\cup {\cal D}_0$ 
is of the form required in Definition \ref{branchedcover}. The first item in part (1) of the proposition
follows from invariance under the action of $\mathbb{C}^*$. 

In the case $m_1\not=m_2$ the tuple of functions $(\hat v_1,\hat v_2,\dots,\hat v_k)$ defines
coordinates on ${\cal Q}_0\cup {\cal D}_0$. The equation $(z-v_1)^{m_1}(z+v_1)^{m_2}=0$ 
is the equation of a union of two complex lines in $\mathbb{C}^2$ 
which intersect transversely in a single point $0$. Thus the union of these
two planes has a normal crossing singularity at  $0$.
As this applies to a neighborhood of the zero of order $m$ in the fiber over any point in 
${\cal D}_0$, it follows that $Z$ is a complex variety with a normal crossing singularity 
at the zero of order $m$. Furthermore, in these coordinates, near the zero of order $m$ 
the projection 
$\Pi_0\vert Z:Z\to {\cal Q}_0\cup {\cal D}_0$  
is of the form required in Definition \ref{branchedcover2}. 
The second item in part (1) of the proposition follows again from invariance under the action of 
$\mathbb{C}^*$.

If the differentials in ${\cal D}$ have $k\geq 2$ zeros of order $m$, then there are $k$ 
sheets for the intersection of ${\cal Q}\cup {\cal D}$ with ${\cal D}$, with  
a normal crossing intersection, and 
the above discussion applies separately to each of these sheets. This shows 
part (2) of the proposition. 
\end{proof}

\section{On the cohomology of 
  the closure of a stratum}\label{thetauto}

In this section we begin the investigation of the second cohomology of 
the closure $\overline{\cal Q}$ of 
a projective stratum ${\cal Q}$ of abelian differentials, and we establish
Theorem \ref{strata} (see \cite{Ch17}). 

As in Section \ref{oversurfaces}, let $P^*{\cal C}$ be the pull-back of the universal curve
${\cal C}\to {\cal M}_g$ 
to ${\cal P}$ and let $\Delta$ be the codimension one subvariety
defined by the zeros of the projective differentials in ${\cal P}$. 
Let $\nu$ be the vertical tangent bundle of $P^*{\cal C}$ and 
let $\tau\to P^*{\cal C}$ be the pull-back of the tautological line bundle 
on ${\cal P}$. 

In the sequel we always denote by $\zeta^*$ the dual of a complex
line bundle $\zeta$, or, equivalently, the inverse of $\zeta$ in the group of 
all complex line bundles on $P^*{\cal C}$. 
We are interested in topological properties of
holomorphic line bundles, that is, in their Chern class. Some of the statements
below also hold in the holomorphic setting. An example is the following

\begin{lemma}\label{trivial2}
The line bundle $\nu\otimes \tau$ is trivial on 
$P^*{\cal C}-\Delta$. 
\end{lemma}
\begin{proof}
Let $\alpha\in \nu^*$ be any
vector in the vertical cotangent bundle of $P^*{\cal C}$ at 
a point $y\in P^*{\cal C}-\Delta$. 
We may view
$\alpha$ as a $\mathbb{C}$-linear 
functional on the holomorphic tangent space $\nu_y$ 
of the fiber of $P^*{\cal C}$ through $y$.

The fiber $\tau_{\Pi(y)}$ of $\tau$ at the point $\Pi(y)\in {\cal P}$ consists
of the line of holomorphic one-forms on the Riemann surface
$P\Pi(y)$ 
in the projective class defined by $\Pi(y)$.
As $y$ is not a zero of a differential in this projective class and 
as the dimension of the
complex vector space of $\mathbb{C}$-linear
functionals $\nu_y\to \mathbb{C}$ equals one,
there is precisely one
holomorphic one-form $\Lambda(\alpha)\in \tau_{\Pi(y)}$ 
whose restriction to $\nu_y$ 
coincides with $\alpha$. Then $\alpha\to \Lambda(\alpha)$ defines an isomorphism
between $\nu^*$ and $\tau$ on $P^*{\cal C}-\Delta$, and hence it defines  
a nowhere vanishing section of the bundle 
$(\nu^*)^*\otimes \tau=\nu\otimes \tau$ 
on $P^*{\cal C}-\Delta$ which is what we wanted to show.
\end{proof}

The codimension one complex subvariety $\Delta\subset P^*{\cal C}$
is a Weil divisor and hence a Cartier divisor in the
smooth complex orbifold $P^*{\cal C}$. Thus it defines a 
holomorphic line bundle $L\to P^*{\cal C}$ whose first Chern class
$c_1(L)$ is dual to $\Delta$ in the sense of intersection (see \cite{Fu84} for
more and for references). This line bundle is trivial on the complement of 
$\Delta$ and restricts to the normal bundle on the regular part of $\Delta$.
Lemma \ref{trivial2} indicates that this line bundle may be a power of
the bundle $\nu\otimes \tau$ in the Picard group of $P^*{\cal C}$. 

Instead of pursuing this line of idea, we identify the cohomology class defined 
by the bundle $L$ which is a weaker statement, but sufficient for our purpose. 
Thus the following statement is meant in the topological sense, and it can be viewed
as a version of Theorem 3.13 of \cite{H20}. 
By the usual exact sequence in sheaf cohomology
defined by the exponential function, it is equivalent to stating 
that the Chern classes of these line bundles coincide.

\begin{proposition}\label{identify}
$\nu^*\otimes \tau^*=L$ on $P^*{\cal C}$.
\end{proposition}
\begin{proof}
  Let $\Sigma$ be a closed oriented surface and let
  $\phi:\Sigma\to P^*{\cal C}$
  be a smooth map. By transversality,
after changing $\phi$ with a homotopy 
 we may assume that
  $\Pi\circ \phi$ intersects ${\cal P}(1)$ in only isolated points and that 
  furthermore, if $x\in \Sigma$ is such that $\Pi\circ \phi(x)\in {\cal P}(1)$
  then $\phi(x)\not\in \Delta$
(see \cite{H20} for a detailed discussion).
Moreover, as $\Delta\cap \Pi^{-1}({\cal P}(0))$ is
the image of a holomorphic 
multisection of the restriction of $\Pi^{-1}({\cal P}(0))$ to ${\cal P}(0)$,  
we may assume that 
$\phi(\Sigma)\cap \Delta$ consists of 
  finitely many transverse intersection points, 
  say the points  $x_1,\dots, x_s$. 
  Each of these points $x_i$ is a simple
   zero of the differential $\Pi(x_i)$.
    It now
  suffices to show that
  $\phi^*(c_1(\nu^*\otimes\tau^*))[\Sigma]$
  equals the number of intersection points of $\phi(\Sigma)$ with
  $\Delta$, counted with sign and multiplicity. Here $[\Sigma]$ denotes the 
  fundamental cycle of $\Sigma$.

Let us without loss of generality assume in addition 
that $\Sigma$ is equipped with a complex structure and that 
the restriction of $\phi$ to a disk neighborhood $D_i$ of $x_i$ in $\Sigma$ 
is a holomorphic or
antiholomorphic embedding of $D_i$ into a fiber of $\Pi$. The latter
can be achived by modifying $\phi$ 
with a small homotopy. 

By Lemma \ref{trivial2}, the restriction of the 
line bundle $\nu^*\otimes \tau^*$ to the complement of 
$\Delta$ admits a natural trivialization 
whose restriction to the boundary $\phi(\partial D_i)$ of the disk $\phi(D_i)$
can be described as follows. 

Choose a trivialization $Y$ of the tangent bundle of $D_i\sim \phi(D_i)$. 
Choose furthermore a trivialization $\xi$ 
of $\nu^*\otimes \tau^*$ over
$\phi(D_i)$. We may assume that the contraction of $\xi$ with $Y$ is constant,
that is, it is the pull-back of a fixed nontrivial vector $q$ in the
fiber of $\tau^*$ over $\Pi \phi(D_i)$.

By Lemma \ref{trivial2} and its proof, 
the contraction of the vector field $Y$ with
the restriction of the trivialization of
$\nu^*\otimes \tau^*$ on $P^*{\cal C}-\Delta$ to 
$\phi(\partial D_i)$ equals the 
section of $\tau^*$ which associates
to a point $p\in \phi(\partial D_i)$ the element of $\tau^*_p$ 
which is defined by
the following linear functional $\zeta_p$. 
Recall 
that the fixed vector $q$ is 
a holomorphic differential on
the fiber of $P^*{\cal C}$ containing $\phi(D_i)$ 
which does not vanish at $p$; we then have 
$\zeta_p(aq)=aq(Y_p)$.

Since $x_i$ is a zero of the differential $q$ and the only zero of $q$ in $D_i$,
if we equip 
$\phi(\partial D_i)$ with the orientation defined by the Riemann surface
structure of the fiber of $P^*{\cal C}$ containing $\phi(D_i)$, then
for this orientation, the map  $S^1=\phi(\partial D_i)\to \mathbb{C}^*$ defined by 
$\zeta_p(q)=q(Y_p)$ has rotation number one. This shows that if 
the restriction of $\phi$ to $D_i$ is holomorphic, then 
the rotation number of the restriction to $\partial D_i$ of 
the trivialization of the bundle $\phi^*(\nu^*\otimes \tau^*)$ on $P^*{\cal C}-\Delta$ 
to $\partial D_i$  with respect to
a trivialization of $\phi^*(\nu^*\otimes \tau^*)$  
on the disk $D_i$ equals one, and it equals $-1$ otherwise. 

As a consequence, the value $c_1(\nu^*\otimes \tau^*)[\Sigma]$ indeed equals the number
of intersections of $\phi(\Sigma)$ with $\Delta$, counted with sign and multiplicities. 
\end{proof}


Let now ${\cal Q}\subset {\cal P}$ be a component of a stratum of projective
abelian differentials, with $m$ zeros of order $k_i$ $(i=1,\dots,s)$.
The closure in $P^*{\cal C}$ of the 
locus of the zeros of order $k_i$ in $\Pi^{-1}{\cal Q}$ is
a complex subvariety $\Delta_{k_i}$ of
$\Delta\cap \Pi^{-1}\overline{\cal Q}$.
Its intersection with $\Pi^{-1}{\cal Q}$ is a holomorphic multisection of
$\Pi^{-1}{\cal Q}$ and hence a smooth complex orbifold. 
We have

\begin{lemma}\label{trivial}
  $(\nu^*)^{\otimes(k_j+1)}\vert \Delta_{k_j}\cap \Pi^{-1}{\cal Q}=
  \tau\vert \Delta_{k_j}\cap \Pi^{-1}{\cal Q}$.
\end{lemma}
\begin{proof} The lemma is well known, and a proof is contained in
  \cite{Ch17}, see also \cite{EKZ}.
  We give a topological proof.

A point $y\in \Delta_{k_j}$ is a zero of order $k_j$ of a 
projective holomorphic one-form on the fiber of $P^*{\cal C}$ containing $y$. 
The fiber $\tau_y$ 
of $\tau$ at $y$ can be identified with the complex line of holomorphic 
one-forms in this projective class. 

As $y$ is a zero of this holomorphic one-form of order $k_j$, there is a holomorphic 
local coordinate $z$ on $\Pi^{-1}(\Pi(y))$ near $y$, with
$y$ corresponding to $z=0$, such that a nonzero differential in the line 
$\tau_y$ can locally near $y$ be written in the form 
$az^{k_j}dz$ for some $a\in \mathbb{C}^*$. 
This differential then defines a singular euclidean metric near 
$y$, which has a cone point of cone angle $2\pi(k_j+1)$ at $y$.

A geodesic arc $\gamma$ in the fiber $\Pi^{-1}(\Pi(y))$ with one endpoint at $y$ and 
no singular point in its interior defines  
a $\mathbb{C}$-valued functional $\beta_\gamma$ on $\tau_y$ by
associating to a differential $\omega\in \tau_y$
the \emph{complex} length of $\gamma$ with respect
to the singular euclidean metric 
defined by $\omega$, that is, we distinguish real and
imaginary part of this length,
and we distinguish the orientation.


If $\gamma$ is not trivial then we have
$\beta_{\gamma^\prime}=\beta_\gamma$  
if and only if $\gamma^\prime$ is obtained from $\gamma$ by a rotation
at $y$ by the angle 
$e^{2\pi i \ell/(k_j+1)}$ for some $\ell\in \mathbb{Z}$ in the complex
coordinate $z$. 
As a consequence, for a nontrivial arc $\gamma$ the
map which associates to $\theta\in S^1$ the functional
defined by the image of $\gamma$ by rotation with angle $\theta$
defines a $k_j+1$-sheeted covering of
the fiber of $\tau^*$ at $y$. As $y\in \Delta_{k_j}$ was arbitrary, this
shows that $\nu\vert \Delta_{k_j}$ is a $k_j+1$-th root of
$\tau^*\vert \Delta_{k_j}$. Equivalently, we have
$\tau\vert \Delta_{k_j}=(\nu^*)^{\otimes (k_j+1)}\vert \Delta_{k_j}$.
As this identification is natural and hence depends in a holomorphic fashion on 
$y\in \Delta_{k_j}$ this shows 
the lemma.
\end{proof}

\begin{remark}\label{consistent}
Let us consider the restriction of the bundle $\nu^*\otimes \tau^*$ to the 
preimage $\Pi^{-1}({\cal Q})$ of a 
stratum ${\cal Q}$ with $m$ zeros of order $k_i$. Proposition 
\ref{identify} shows that the  restriction of $\nu^*\otimes \tau^*$ 
to $\Pi^{-1}{\cal Q}-\Delta$ is trivial, and 
taking the tensor product of the equation in Lemma \ref{trivial} with $\nu$ and 
dualizing yields that its restriction to $\Delta_{k_j}$ equals the restriction of 
$\nu^{k_j}$. This is consistent as the restriction of a holomorphic line bundle
to a defining divisor equals the normal bundle of the divisor, and since $\Delta\cap 
\Pi^{-1}{\cal Q}$ is a multisection of $\Pi^{-1}{\cal Q}$ over ${\cal Q}$, this normal
bundle equals the vertical tangent bundle $\nu$. Furthermore, the 
multiplicity of $\Delta_{k_j}$ in $\Delta\cap \Pi^{-1}{\cal Q}$ equals $k_j$.
\end{remark}

We use the results obtained so far to show Theorem \ref{strata} from the introduction
(see \cite{Ch17} and also \cite{EKZ}). To this end
recall that 
the $\ell$-th Mumford Morita Miller class
$\kappa_\ell\in H^{2\ell}({\cal M}_g,\mathbb{Q})$ 
is defined as follows \cite{M87}. Let as before 
$\Upsilon:{\cal C}\to 
{\cal M}_g$ be the universal curve, let $B$ be closed oriented manifold
and let $\phi:B\to {\cal M}_g$ be a smooth map; 
then $E=\phi^*{\cal C}$ is a surface bundle over $B$ with vertical tangent bundle
$\nu$.  We have 
\[\kappa_\ell(\phi(B))=\Upsilon_*(c_1(\nu)^\ell)(\phi(B))\]
where $\Upsilon_*$ is the Gysin push-forward map obtained
by integration over the fiber. 

The following proposition treats the case $\ell=1$ and is included here to make
the argument more transparent.

\begin{proposition}\label{degree2}
  Let ${\cal Q}$ be a component of a stratum
  $\mathbb{P}{\cal H}(\ell_1,\dots,\ell_m)$ of projective abelian differentials
(here the $\ell_j$ are counted with multiplicity); then   
$P^*\kappa_1\vert{\cal Q}=\sum_j(\ell_j+
1-\frac{1}{\ell_j+1})\eta\vert {\cal Q}$.
\end{proposition}
\begin{proof} It suffices to evaluate $P^*\kappa_1$ on the image of 
a smooth map $\phi:B\to {Q}$ where $B$ is a closed oriented surface.

To this end let $1\leq k_1<\cdots <k_m$ be the distinct orders of
the zeros of the differentials in ${\cal Q}$, and let
$d_i\geq 1$ be the multiplicity of the zero of order $k_i$.
Let $\Pi^E:E\to B$ be the surface bundle $(P\circ \phi)^*{\cal C}$.
The hypersurface $\Delta_{k_j}$ in $P^*{\cal C}$ pulls back to a smooth 
multisection of $E\to B$ which defines a homology class
$\delta_{k_j}\in H_2(E,\mathbb{Q})$. Write 
$\delta=\sum_jk_j\delta_{k_j}\in H_2(E,\mathbb{Q})$.

By Proposition \ref{identify} and naturality of Chern classes under
pull-back by inclusions, 
the first Chern class 
$\phi^*(c_1(\nu^*)-c_1(\tau))$ of the pull-back  
bundle $\phi^*(\nu^*\otimes \tau^*)$ is Poincar\'e dual to 
the homology class $\delta$. Denoting again by $\nu^*$ the vertical cotangent
bundle of $E\to B$ and 
omitting the pull-back by $\phi$ in our notation, we have 
\[(c_1(\nu^*)-c_1(\tau))\cup \xi[E]=\xi(\delta)\]
for every $\xi\in H^2(E,\mathbb{Q})$, where $[E],[B]$ is the fundamental cycle
of $E,B$.

As a consequence, we compute (see \cite{H20} for details) 
\begin{align}\label{expansion1}
\phi^*P^*\kappa_1[B]& =c_1(\nu^*)\cup c_1(\nu^*)[E] \\
 & =(c_1(\nu^*)-c_1(\tau))\cup c_1(\nu^*)[E]+c_1(\tau)\cup c_1(\nu^*)[E]\notag\\
   & =c_1(\nu^*)(\delta)+c_1(\tau)\cup c_1(\nu^*)[E].\notag
    \end{align}

By Lemma \ref{trivial},  the restriction of $(\nu^*)^{\otimes (k_j+1)}$
to $\Delta_{k_j}$ is equivalent to the restriction of $\tau$
and therefore 
\begin{equation}\label{expansion2}
c_1(\nu^*)(\delta_{k_j})=\frac{1}{k_j+1}c_1(\tau)(\delta_{k_j}).
\notag\end{equation}

The restriction 
$\Pi\vert \Delta_{k_j}\cap \Pi^{-1}{\cal Q}:\Delta_{k_j}\cap \Pi^{-1}{\cal Q}
\to {\cal Q}$ is an (unbranched) covering 
of degree $d_j$. Since $c_1(\tau)$ is the pull-back to $E$ of the class
$\phi^*(\eta)\in H^2({\cal Q},\mathbb{Q})$, we have 
\[c_1(\tau)(\delta_{k_j})=d_j \eta(\phi_*[B]).\]
Thus by the definition of $\delta$, we have 
\begin{equation}\label{expansion3} 
  c_1(\nu^*)(\delta)=\sum_j\frac{k_jd_j}{k_j+1}\eta(\phi_*[B]).
\end{equation}

Using once more that
$c_1(\tau)$ is the pull-back of the class $\phi^*(\eta)$ on $B$ and that the evaluation
of $c_1(\nu^*)$ on a fiber of $P^*{\cal C}$ equals $2g-2$,  
we also have
\begin{equation}\label{expansion4}
  c_1(\nu^*)\cup c_1(\tau)[E]=(2g-2)\eta(\phi_*[B]).\end{equation} 
Now $\sum_jk_jd_j=2g-2$ and hence we conclude from
equations (\ref{expansion1}), (\ref{expansion3}) and (\ref{expansion4})
that 
\begin{align}
  \kappa_1(\phi_*[B])&
=\sum_j(1-\frac{1}{k_j+1})d_j\eta(\phi_*[B])+\sum_jk_jd_j\eta(\phi_*[B]) \notag\\
 & = \sum_j(k_j+1-\frac{1}{k_j+1})d_j \eta (\phi_*[B]).\notag \end{align}
Since the degree $d_j$ equals the number of zeros of order $k_j$ which
are contained in $\Delta_{k_j}\cap \Pi^{-1}(y)$ for any $y\in B$,
this concludes the proof of the proposition.
\end{proof}

\begin{example}
Let us consider the principal stratum
${\cal Q}={\cal P}-{\cal P}_1$. 
Then $k_j=1$ for all $j$ and hence Proposition \ref{degree2} shows that
\[P^*\kappa_1\vert {\cal Q}=\sum_j(2-\frac{1}{2})\eta\vert {\cal Q}=
  (3g-3)\eta\vert {\cal Q},\]
which is consistent with 
Theorem \ref{kozo}. 
\end{example}

\begin{proof}[Proof of Theorem \ref{strata}]
Let ${\cal Q}\subset {\cal P}$ be a component of a stratum of 
projective abelian differentials with $m$ zeros of order $k_i$.   
It suffices to evaluate $P^*\kappa_\ell$ for $\ell\geq 1$ on the image of 
a finite $2\ell$-dimensional simplicial 
Poincar\'e duality complex
$B$ of homogeneous dimension $2\ell$ 
under a continuous map $\phi:B\to {\cal Q}$.

To this end consider the surface bundle $\Pi:E\to B$ defined by $P\circ \phi$.
Let $[E]$ be the fundamental class of $E$.
The zeros of the differentials in $\phi(B)$ of order $k_i$ define
define a multisection $\Delta_{k_i}^E$ of $E$, and each of these
multisections defines a homology class $\delta_{k_i}\in H_{2\ell}(E,\mathbb{Q})$.
By Lemma \ref{dual},  the homology class
$\delta=\sum_jk_j \delta_{k_j}$ is Poincar\*e dual to
$c_1(\nu^*)-c_1(\tau)$ (compare the discussion in the proof of Proposition \ref{degree2}
which carries over without change; here we omit in our notation that
all classes on $E$ are pull-backs of classes on $P^*{\cal C}$).
In particular, 
by the definition of the $\ell$-th Mumford Morita Miller
class $\kappa_{\ell}$ \cite{M87} and the fact that $c_1(\tau)^{\ell+1}[E]=0$ since
$c_1(\tau)$ is the pull-back of a cohomology class on ${\cal Q}$,
using the Ansatz in the proof of Proposition \ref{degree2} we have
\begin{align}\label{kappak1}
(-1)^{\ell+1}  \kappa_\ell(P\phi_*[B]) &=c_1(\nu^*)^{\ell+1}[E]   \\ &=
 (c_1(\nu^*)-c_1(\tau))\cup c_1(\nu^*)^\ell[E]+
  c_1(\tau)\cup c_1(\nu^*)^\ell[E]. \notag \end{align}

Taking into account the fact that all contributions are of degree two and hence
their cup products commute, we can 
expand the second summand in this equation as 
\begin{align}\label{kappak2}
 c_1(\tau)\cup & c_1(\nu^*)^\ell[E]  \\ 
&=(c_1(\tau)\cup (c_1(\nu^*)-c_1(\tau))\cup c_1(\nu^*)^{\ell-1}[E] +c_1(\tau)^2\cup c_1(\nu^*)^{\ell-1}[E]\notag \\
&=(c_1(\nu^*)-c_1(\tau))\cup c_1(\tau)\cup c_1(\nu^*)^{\ell-1}[E] +c_1(\tau)^2\cup c_1(\nu^*)^{\ell-1}[E].\notag
\end{align}  
Proceeding inductively, we obtain the equation
\begin{align}\label{kappak3}
 (-1)^{\ell +1} &\kappa_\ell(P\phi_*[B])\\ 
&= \sum_{j=0}^{\ell}(c_1(\nu^*)-c_1(\tau))
                       \cup \bigl( c_1(\tau)^j \cup c_1(\nu^*)^{\ell-j}\bigr) [E]
       +c_1(\tau)^{\ell+1}[E] \notag \\            
&=\sum_{j=0}^{\ell}c_1(\tau)^j\cup c_1(\nu^*)^{\ell-j}(\delta).\notag
\end{align}

The restriction of $\Pi$ to each of the sets $\Delta^E_{k_i}\subset \Delta^E$
is a covering. 
Moreover, by Lemma \ref{trivial}, we have
$(\nu^*)^{\otimes (k_i+1)}\vert \Delta_{k_i}^E=\tau\vert \Delta_{k_i}^E$.
This yields 
\[c_1(\tau)^j\cup c_1(\nu^*)^{\ell-j}(\delta_{k_i})=\frac{1}{(k_i+1)^{\ell-j}}c_1(\tau)^\ell(\delta_{k_i})\]
and therefore 
\[\sum_{j=0}^\ell c_1(\tau)^j\cup c_1(\nu^*)^{\ell-j}(\delta)=
\sum_i \bigl( \sum_{j=0}^{\ell} 
\frac{k_i}{(k_i+1)^{\ell-j}} \bigr) c_1(\tau)^\ell(\delta_{k_i})\]
and hence the theorem follows from the fact that $c_1(\tau)$ is the pull-back of
the class $\phi^*(\eta)\in H^2({\cal Q},\mathbb{Q})$ 
and that furthermore the restriction of 
the projection $\Pi$ to $\Delta_{k_i}^E$ is an unbranched covering of degree $d_i$.
\end{proof}

 \section{The Poincar\'e dual of ${\cal P}_1$}
\label{signatureasintersection}

In this section we apply the results of
Section \ref{oversurfaces} 
to represent the signature of a surface bundle as an intersection number. 
This yields a purely topological proof of Theorem \ref{kozo}
\cite{KZ11}.

The projectivized Hodge bundle $P:{\cal P}\to {\cal M}_g$
extends to a bundle over the Deligne Mumford 
compactification $\overline{\cal M}_g$ of ${\cal M}_g$
which we denote by $P:\overline{\cal P}\to \overline{\cal M}_g$. 
A standard spectral sequence argument shows
that the second rational cohomology group of 
$\overline{\cal P}$ is generated by the pull-back of 
the second rational 
cohomology group of $\overline{\cal M}_g$ together with the cohomology
class $\eta$ of the tautological line bundle of the fibre (see
Lemma 1 of \cite{KZ11} for details).
Since $\overline{\cal P}$ 
is a Poincar\'e duality space,
there is a cohomology class
$\xi\in H^2(\overline{\cal P},\mathbb{Q})$ which is Poincar\'e dual
to the closure $\overline{{\cal P}(1)}$ of 
${\cal P}(1)$. The set $\overline{{\cal P}(1)}$ is in fact a (singular)
complex hypersurface in $\overline{\cal P}$.

The 
class $\xi$ can be expressed as a rational linear combination
of the class $\eta$ and the pull-back of a set of  generators of 
$H^2(\overline{\cal M}_g,\mathbb{Q})$. Such a set of generators consists
of the first Chern class $\lambda$ of the Hodge bundle 
as well as the Poincar\'e duals $\delta_j$ $(0\leq j\leq \lfloor g/2\rfloor)$
of the irreducible components of the boundary divisor $\overline{\cal M}_g-{\cal M}_g$.
We refer to \cite{HaM98} for more information.

The first Chern class $\lambda$ of the Hodge bundle
is non-zero precisely when $g\geq 3$ \cite{H83,HaM98}.
The class 
$\delta_0$ is dual to the divisor of stable curves with a single non-separating node, and
for $1\leq j\leq g/2$, the class $\delta_j$ is dual to the divisor of stable curves with 
a node separating the stable curve into a curve of genus $j$ and
a curve of genus $g-j$. 

Korotkin and Zograf 
calculated this linear combination
(the formula before Remark 2 on p.456
of \cite{KZ11}) using ideas from mathematical physics. 
An algebraic geometric 
proof of Theorem \ref{kappaone} 
is due to Chen \cite{Ch13}.

\begin{theorem}[Korotkin and Zograf \cite{KZ11}]\label{kappaone}
\[\xi=24P^*\lambda-(6g-6)\eta-P^*(2\delta_0-3\sum_{j=1}^{\lfloor g/2\rfloor} \delta_j).\] 
\end{theorem}

Let $\Pi^E:E\to B$ be a surface bundle over a surface, defined by 
a smooth map $f:B\to {\cal M}_g$. Let ${\cal S}\to {\cal M}_g$
be the sphere subbundle of the Hodge bundle over ${\cal M}_g$;
it admits a fibration $\Xi:{\cal S}\to {\cal P}$ with fiber a circle. 
Let $F:B\to {\cal S}$ be a lift of $f$
to ${\cal S}$.
Such a lift exists since the dimension of the fiber of ${\cal S}$ equals
$2g-1>2$ (see for example Lemma 3.5 of \cite{H20}). 
Denote by $\Delta_E\subset F^*\Xi^*P^*{\cal C}$ the pull-back of the variety 
$\Delta$ in the pull-back of the universal curve to $B$.
Then $\Delta_E$ defines a homology
class $[\Delta_E]\in H_2(E,\mathbb{Q})$ which is Poincar\'e dual to
the Chern class $c_1(\nu^*)$
of the vertical cotangent of $E$ 
by Proposition \ref{identify} and the fact that 
the pull-back to ${\cal S}$ of the tautological bundle on ${\cal P}$ 
is trivial. We refer to Section \ref{thetauto} for a more detailed discussion.
Denoting as before by $[E]$ the fundamental class of $E$ we have

\begin{corollary}\label{selfinter2}
$c_1(\nu^*)\cup c_1(\nu^*)[E]=[\Delta_E]\cdot [\Delta_E]=
c_1(\nu^*)[\Delta_E]$.
\end{corollary}
\begin{proof}
Both equations follows from Poincar\'e duality 
for the surface bundle $E$.
\end{proof}

By the definition of the first Mumford Morita Miller class $\kappa_1$ as 
explained before Proposition \ref{degree2}, we obtain

\begin{corollary}\label{tau}
$f^*\kappa_1[B]=c_1(\nu^*)[\Delta_E].$
\end{corollary}

In view of the identity $\kappa_1=12\lambda$ on ${\cal M}_g$
\cite{HaM98}, the formula in Theorem \ref{kozo} is a special case of 
Theorem \ref{kappaone}. The following lemma is
the first step towards a purely topological proof.
For its formulation, recall that we can look at the
intersection number between $\phi(B)$ and the hypersurface
${\cal P}(1)\subset {\cal P}$.

\begin{lemma}\label{intersection1}
$\Xi F(B)\cdot {\cal P}(1)=2[\Delta_E]\cdot[\Delta_E]=2c_1(\nu^*)[\Delta_E]$;
in particular, the restriction of the class $\xi$ to 
${\cal P}\vert {\cal M}_g$ satisfies 
\[\xi=2P^*\kappa_1+ a\eta=24 P^*\lambda
+a\eta\] for some $a\in \mathbb{Q}$. 
\end{lemma}
\begin{proof}  Since the complex codimension of ${\cal P}(2)\subset {\cal P}$ equals 
two and since by Proposition \ref{addboundarycomponent} the union
${\cal P}(0)\cup {\cal P}(1)$ is a smooth orbifold (recall to this end that 
${\cal P}(1)$ is just the stratum of abelian differentials with a single zero of order two and
all other zeros of order one), by transversality we may assume that 
$\Xi F(B)\subset {\cal P}(0)\cup {\cal P}(1)$ and that furthermore  $\Xi F(B)$ intersects
${\cal P}(1)$ transversely in finitely many points. 

Let as before $\Delta_E\subset E=F^*\Xi^*P^*{\cal C}$ 
be the pull-back of $\Delta$ in the surface bundle $E\to B$. 
By the first part of Proposition \ref{boundarycomp2}, 
$\Delta_E$ is a smoothly embedded
surface in $E$. 
Furthermore, the restriction of the projection $\Pi^E:E\to B$
to $\Delta_E$ is a branched covering, doubly branched at each double zero 
of a differential in the finitely many intersection points of 
$\Xi F(B)$ with ${\cal P}(1)$.

At each branch point $x\in \Delta_E$, the surface $\Delta_E$ is
tangent to the fiber of $E$ at $x$, and the orientation of
$\Delta_E$ coincides with the orientation of the fiber if and only
if the intersection point $\Xi F(x)$ of $\Xi F(B)$ with ${\cal P}(1)$ 
is positive.

Assigning to each branch point in $\Delta_E$ this sign
defines a divisor $A$ on the surface $\Delta_E$. 
The tangent bundle of $\Delta_E$ 
can be represented in the form 
$(\Pi_E\vert \Delta_E)^*(TB)\otimes(-H)$ where
$H$ is the line bundle on $\Delta_E$ with 
divisor $A$. Thus the 
normal bundle $N$ of $\Delta_E$ can
be written as $N=\nu\otimes H^+(\otimes H^-)^{-1}$ where $H^+$
is the line bundle defined by the divisor on $\Delta_E$ 
which corresponds
to the positive intersection points of $\Xi F(B)$ 
with ${\cal P}(1)$, and $H^-$ is the line bundle defined
by the divisor on $\Delta_E$ which corresponds to the negative
intersection points. 

This implies that the
self-intersection 
number in $E$ of the surface $\Delta_E\subset E$ 
equals
\[[\Delta_E]\cdot [\Delta_E]=c_1(\nu)[\Delta_E]+b\] where
$b=\Xi F(B)\cdot {\cal P}(1)$ 
is the number of branch points of $\Pi_E\vert \Delta_E$, counted with sign.

By Poincar\'e duality (see Corollary \ref{selfinter2}),
we have
\[c_1(\nu^*)[\Delta_E]=[\Delta_E]\cdot [\Delta_E]=
c_1(\nu)[\Delta_E]+b=-c_1(\nu^*)[\Delta_E]+b\]
and hence $b=2c_1(\nu^*)[\Delta_E]$. Together with Corollary \ref{tau}
and the fact that $\kappa_1=12\lambda$ as classes in 
$H^2({\cal M}_g,\mathbb{Q})$ \cite{HaM98}, 
this completes the proof of the lemma.
\end{proof}

For the proof of Theorem \ref{kozo} we are left with computing the constant $a\in \mathbb{Q}$.

\begin{proof}[Proof of Theorem \ref{kozo}] 
To calculate the coefficient $a\in \mathbb{Q}$ in the
expression in Lemma \ref{intersection1} note first 
that in the case $g=2$, we have $\lambda=0$ \cite{HaM98} and 
\[\xi=2P^*\kappa_1-6\eta=-6\eta.\] 

Namely,  for $g=2$ 
the complex rank of the Hodge bundle equals 2
and hence the fibre of the bundle ${\cal P}\to {\cal M}_2$ over the moduli space
of genus 2 complex curves is just $\mathbb{C}P^1$. 
A \emph{Weierstrass point} on a genus 2 complex curve $X$ is a double zero of a holomorphic
one-form on $X$. Now
$X$ has precisely $6=-3\chi(S_2)$ Weierstrass points and hence
the intersection number of the fibre of the bundle ${\cal P}\to {\cal M}_2$ 
with the divisor
${\cal P}(1) $ equals $6$. As the evaluation on $\mathbb{C}P^1$ of the 
Chern class of the
tautological line bundle on $\mathbb{C}P^1$
equals $-1$, the formula in the theorem follows from Poincar\'e duality. 

For arbitrary $g\geq 3$ 
choose a complex curve $X\in {\cal M}_g$ 
which admits an unbranched cover of degree $g-1$ onto a curve
$Y\in {\cal M}_2$.
The projective line 
of projective holomorphic one-forms on $Y$ 
pulls back to a projective line of projective holomorphic one-forms on $X$.
The pull-back of a projective differential with two simple zeros is a 
differential with only simple zeros, but the pull-back $q$ of a differential
with a double zero is a differential
with $g-1$ double zeros.
By Proposition \ref{boundarycomp2}, such a differential is contained in
a component ${\cal D}$ of a  stratum of differentials with $g-1$ double zeros, 
and it is the locus of a $g-1$-fold
normal crossing of its union with the (connecting) stratum
${\cal Q}=\mathbb{P}{\cal H}(1,1,2,\dots,2)$.

The projective line $\mathbb{C}P^1$ 
of holomorphic one-forms pulled back from $Y$ is transverse at $q$ 
to each of the $d$ local branches through $q$ of the closure of 
${\cal Q}$.  As a consequence, if we fix a
small disk $D$ about $q$ in the fiber $\mathbb{C}P^1$ of the bundle
${\cal P}$, then the homological 
intersection number of $(D,\partial D)$ with each of 
these $d$ local branches, 
that is, the intersection number of 
a deformation of $D$ with fixed boundary, counted with sign and 
multiplicities, equals one. 
As a consequence, the intersection number with
${\cal P}(1)$ of this pulled-back $\mathbb{C}P^1$ equals 
$6d=6g-6$. This 
completes the proof of Theorem \ref{kozo}.  
\end{proof}

\begin{remark}
Theorem \ref{kozo} also shows the following. Let $X$ be a Riemann surface
of genus $g$. Then the intersection of 
$\overline{{\cal P}(1)}$ with the complex projective space 
$\mathbb{C}P^{g-1}$ 
defined as the projectivization of the vector space of holomorphic one-forms on $X$ 
is a complex hypersurface in $\mathbb{C}P^{g-1}$. The degree of this hypersurface 
equals $6g-6$.
\end{remark}

\section{Boundary divisor computation}\label{boundary}

Let ${\cal Q}$ be a component of a stratum of projective
abelian differentials with at least two zeros,
with closure $\overline{\cal Q}$. 
Assume that the differential in ${\cal Q}$ have $d_i$ zeros of
order $k_i$. These zeros
define a multisection of the restriction of
$P^*{\cal C}\vert {\cal Q}=\Pi^{-1}{\cal Q}$ to ${\cal Q}$
whose closure in 
$P^*{\cal C}$ will be denoted by 
$\Delta_{k_i}$.
The goal of this section is to use the hypersurfaces
$\Delta_{k_i}\subset \Pi^{-1}\overline{\cal Q}$
to obtain some information on the second 
cohomology group of $\overline{\cal Q}$.

The following is immediate from Proposition \ref{variety} and 
Proposition \ref{boundarycomp2}.

\begin{lemma}\label{cycle}
$\Delta_{k_i}$ is a codimension one subvariety
of $P^*{\cal C}\vert \overline{\cal Q}$ which is a local complete intersection.
Thus $\Delta_{k_i}$ defines a class in
$H^2(P^*{\cal C}\vert \overline{\cal Q},\mathbb{Q})$. 
\end{lemma}
\begin{proof} By Proposition \ref{variety} and
  Proposition \ref{boundarycomp2},
  $\Delta_{k_i}$ is a codimension one subvariety of the
variety  $P^*{\cal C}\vert \overline{\cal Q}$ which is a local 
complete intersection.
Thus $\Delta_{k_i}$ is a 
Weil divisor in $P^*{\cal C}\vert \overline{\cal Q}$ and hence a Cartier 
divisor since $P^*{\cal C}\vert \overline{\cal Q}$ is a complex variety. 
Via intersection, 
such a divisor defines a class in
$H^2(P^*{\cal C}\vert \overline{\cal Q},\mathbb{Q})$ \cite{Fu84}.
\end{proof}

Denote by $L_{k_i}$ the holomorphic line bundle which defines $\Delta_{k_i}$,
that is, so that $\Delta_{k_i}$ is the zero set of a rational section of
$L_{k_i}$. On the regular subset of $\Delta_{k_i}$, the restriction of 
$L_{k_i}$ coincides with the normal 
bundle of $\Delta_{k_i}$. Since $\Delta_{k_i}\cap 
\Pi^{-1}({\cal Q})$ is 
a holomorphic multi-section of $\Pi^{-1}{\cal Q}\to {\cal Q}$,
the line bundle $L_{k_i}$ has the following properties.
\begin{enumerate}
\item $L_{k_i}\vert \Pi^{-1}\overline{\cal Q}-\Delta_{k_i}$ is trivial.
\item The restriction of $L_{k_i}$ to $\Delta_{k_i}\cap \Pi^{-1}({\cal Q})$ 
coincides with the
restriction of vertical tangent bundle $\nu$. 
\item 
The degree of the restriction of $L_{k_i}$ to a fiber of the bundle 
$\Pi^{-1}\overline{\cal Q}$ equals 
the multiplicity $d_i$ of the zero $k_i$.
\end{enumerate}

The last property follows from the fact that the restriction of 
the bundle $L_{k_i}$ to a fiber $X$ of
$\Pi^{-1}\overline{\cal Q}$ is the line bundle on $X$ defined
by the effective divisor
$\Delta_{k_i}\cap X$, and the degree of this divisor 
equals $d_i$.

The tautological line bundle over the fibers of ${\cal P}$ 
pulls back via the projection $P$ to a line bundle
$\tau$ on $P^*{\cal C}$. 
The following is an easy consequence of Proposition \ref{identify}.

\begin{lemma}\label{dual}
$c_1(\nu^*\otimes \tau^*)\vert \Pi^{-1}\overline{\cal Q}=
\sum_i k_i c_1(L_{k_i})$. 
\end{lemma}
\begin{proof}
  The Weil divisor $\sum_ik_i\Delta_{k_i}$
 is the intersection of the hypersurface 
$\Delta\subset P^*{\cal C}$ with $\Pi^{-1}(\overline{\cal Q})$. 
As the cohomology class dual to a
divisor is defined by intersection \cite{Fu84},
the intersection $\Delta\cap \Pi^{-1}\overline{\cal Q}$, counted
with multiplicites,  
defines the restriction to $\Pi^{-1}\overline{\cal Q}$ 
of the cohomology class dual to $\Delta$.
By Proposition \ref{identify}, this cohomology class is the class
$c_1(\nu^*\otimes \tau^*)$.

On the other hand, since all constructions are
natural, $\sum_ik_i\Delta_{k_i}$ also defines the comology class
$\sum_ik_i c_1(L_{k_i})$. This is what we wanted to show.  
\end{proof}

Let for the moment $X,Y$ be 
arbitrary connected locally path connected 
Hausdorff spaces and let
$\phi:X\to Y$ be an open and closed continuous map. 
The \emph{degree} of $\phi$ is defined as 
\[{\rm deg}(\phi)=\sup\{\sharp \phi^{-1}(y)\mid y\in Y\},\]
and the \emph{local degree} of $\phi$ at $x\in X$ is defined  
as 
\[{\rm deg}(\phi,x)=\inf_U\sup \{\sharp \phi^{-1}\phi(z)\cap U\mid z\in U\},\]
where $U$ ranges over the neighborhoods of $x$.
The following is taken from \cite{Ed76}.

\begin{definition}\label{finitebranched}
An open and closed continuous 
map $\phi:X\to Y$ is a \emph{finite branched covering} if 
${\rm deg}(\phi)<\infty$ and for each $y\in Y$, 
\[{\rm deg}(\phi)=\sum_{x\in \phi^{-1}(y)} {\rm deg}(\phi,x).\]
\end{definition}

The relevance for our purpose is Theorem 2.1 of \cite{Ed76}.

\begin{theorem}[Edmonds \cite{Ed76}]\label{edmonds} 
Let $f:X\to Y$ be a finite branched covering. Then there is a transfer
homomorphism 
\[\tau:H^*(X,\mathbb{Q})\to H^*(Y,\mathbb{Q})\]
such that $\tau\circ  f^*= {\rm deg}(f)\cdot 1$.
\end{theorem}

As in \cite{H20}, we have

\begin{lemma}\label{finitebranchedlem}
  For each $i$ the restriction of the projection $\Pi$ to $\Delta_{k_i}$
  is a finite branched covering.
\end{lemma}
\begin{proof} The restriction of the projection 
$\Pi$ to $\Delta_{k_i}\subset \Pi^{-1}\overline{\cal Q}$ is 
clearly open and closed, and its restriction to 
$\Delta_{k_i}\cap \Pi^{-1}{\cal Q}$ is an unbranched covering of degree $d_i$.
Thus it suffices to observe the following. 
Let $z\in \Delta_{k_i}$; then the local degree of $\phi=\Pi\vert \Delta_{k_i}$
at $z$ equals the multiplicity of $z$ in the divisor supported in 
$\Pi^{-1}(\Pi(z))\cap \Delta_{k_i}$ which defines the restriction of the line 
bundle $L_{k_i}$ to $\Pi^{-1}(\Pi(z))$.

To this end choose  a sequence of points
$z_j\in \Delta_{k_i}\cap \Pi^{-1}{\cal Q}$ such that
  $z_j\to z$ and hence $q_j=\Pi(z_j)\to q=\Pi(z)$.

  For each $j$ the intersection $\Delta_{k_i}\cap 
  \Pi^{-1}(q_j)$ is an effective divisor $D_j$ 
  of degree $d_i$ where
  $d_i$ is the multiplicity of the zero of order $k_i$ for $q_j$. By the
  definition of the topology on ${\cal P}$, as $i\to \infty$ 
  these divisors converge to an effective divisor $D$  
  of degree $d_i$ containing $z$. 
  If the multiplicity of $z$ in $D$ equals
$m\geq 1$ then it follows as in Lemma 3.2 of \cite{H20} that the local degree of
$\Pi\vert \Delta_{k_i}$ at $z$ equals $m$. 
This implies the lemma.
\end{proof}
  
By Lemma \ref{finitebranchedlem} and Theorem \ref{edmonds}, 
there is a transfer map
\[H^*(\Delta_{k_i},\mathbb{Q})\to 
H^*({\cal P},\mathbb{Q}).\] Thus 
we can define 
a cohomology class 
\[\kappa^{k_i}\in H^{2}(\overline{\cal Q},\mathbb{Q})\]
as the image of $c_1(L_{k_i})\vert \Delta_{k_i}$ under the 
transfer map $H^*(\Delta_{k_i},\mathbb{Q})\to
H^*(\overline{\cal Q},\mathbb{Q})$.
In other words, if $A$ is cycle in
$\overline{\cal Q}$
which defines a homology class $[A]\in H_2(\overline{\cal Q},\mathbb{Q})$,
then $\kappa^{k_i}[A]$ is the evaluation of
$c_1(L_{k_i})$ on $\Pi^{-1}(A)\cap \Delta_{k_i}$ as 
specified in Theorem \ref{edmonds}.

Denote by 
${\cal D}_{i,j}^c$ the closure of a connected component of the 
 boundary divisor
${\cal D}_{i,j}$ of ${\cal Q}$ which is obtained by colliding two zeros 
of not necessarily distinct order $k_i,k_j$. Recall to this end  
from \cite{KtZ03} that a stratum may have several connected components, and 
the boundary of a component of a stratum may contain
more than one of these components of codimension one.
By Proposition \ref{addboundarycomponent}, 
if differentials in ${\cal Q}$ do not have a zero of order $k_i+k_j$, then 
differential in ${\cal D}_{i,j}^c$ contain a single zero of order $k_i+k_j$, and 
the boundary component ${\cal D}_{i,j}^c$ of ${\cal Q}$ 
is a smooth suborbifold of ${\cal Q}\cup {\cal D}$. Otherwise it is
the locus of a normal crossing singularity.
Thus the closure of such a boundary component
is a divisor in $\overline{\cal Q}$ which 
defines a dual line bundle on $\overline{\cal Q}$. Our goal
is to compute the Chern class of this line bundle
using the classes
$\kappa^{k_\ell}$. 

We begin with 
computing the normal bundle of $\Delta_{k_\ell}$ in 
${\cal Q}\cup {\cal D}_{i,j}^c$.
The following is similar to Lemma \ref{intersection1}.

\begin{lemma}\label{normalbundle}
Assume that the multiplicity of the zero of order $k_i+k_j$ in 
${\cal D}_{i,j}^c$ equals one. Then 
  $\Delta_{k_\ell}$ is a smooth suborbifold of
  $\Pi^{-1}({\cal Q}\cup {\cal D}_{i,j}^c)$.
\begin{enumerate}
\item If $i=j$ and $\ell=i$ then the 
normal bundle of $\Delta_{k_i}$ equals 
$(\nu\vert \Delta_{k_i})\otimes N$ where $N$ is the line
bundle on $\Delta_{k_i}$ defined by the divisor which equals 
the locus of the zeros of order $2k_i$ of the points in 
${\cal D}_{i,i}^c$.
\item If $i=j$ and $\ell\not=i$ or if $i\not=j$ then the normal 
 bundle of $\Delta_{k_\ell}$  
 equals the restriction of 
$\nu$. 
\end{enumerate}
\end{lemma}
\begin{proof}
Recall that for all $\ell$, 
  the normal bundle of the intersection 
$\Delta_{k_\ell}\cap \Pi^{-1}({\cal Q})$ equals the bundle $\nu$. 

Now if $i=j$ then by Proposition \ref{boundarycomp2}, the intersection 
$\Delta_{k_i}\cap \Pi^{-1}({\cal Q}\cup {\cal D}_{i,j}^c)$
is tangent to the fibers of the bundle $P^*{\cal C}$ 
at the zeros of order $2k_i$ in $\Delta_{k_i}$. The locus of these
zeros equals the branch locus
of the projection $\Delta_{k_i}\to {\cal Q}\cup {\cal D}_{i,j}^c$, and it   
is a subvariety of $\Delta_{k_i}$ of codimension one. Furthermore, the projection
$\Pi\vert \Delta_{k_i}$ is doubly branched along its branch locus. 
Therefore the normal bundle of $\Delta_{k_i}$ equals the tensor product of
the restriction of the vertical tangent bundle
$\nu$ with the line bundle which is 
dual to this locus of tangency. We refer to the proof of 
Lemma \ref{intersection1} for more details on this well known fact, with
a slightly different but equivalent viewpoint.

If $i=j$ and $\ell\not=i$ then the restriction of $\Pi$ to
$\Delta_\ell\cap \Pi^{-1}({\cal Q}\cup {\cal D}_{i,j}^c)$
is an unbranched covering and hence the normal bundle of 
$\Delta_\ell\cap \Pi^{-1}({\cal Q}\cup {\cal D}_{i,j}^c)$
coincides with the vertical tangent bundle $\nu$.

If $i\not=j$ then for each $\ell$ the intersection
$\Delta_{k_\ell}\cap \Pi^{-1}({\cal Q}\cup {\cal D}_{i,j}^c)$
is smooth, with normal bundle $\nu$.
Note however that the hypersurfaces
$\Delta_{k_i}$ and $\Delta_{k_j}$ intersect 
along the zeros of order $k_i+k_j$.
\end{proof}

To keep notations transparent, from now
on we denote by
$[A]$ the second cohomology class 
defined by a Cartier divisor $A$. In particular, 
for a component 
${\cal Q}$ of a stratum, with zeros of order $k_i,k_j$,
we obtain a class
$[{\cal D}_{i,j}^c]\in H^2({\cal Q}\cup {\cal D}_{i,i}^c, \mathbb{Q})$
defined by a boundary component ${\cal D}_{i,j}^c$ of ${\cal Q}$
which is obtained 
by merging a zero of order $k_i$ with a zero of order $k_j$. 
Let as before $\tau$ be the pull-back of the tautological bundle 
on ${\cal P}$ to $P^*{\cal C}$. The following statement is the main
technical result required for the proof of Theorem \ref{stratificationthm}, and
it is of independent interest. 

\begin{proposition}\label{tauto}
For all $i,j$ the cohomology class $[{\cal D}_{i,j}^c]$ is contained in the subgroup
of $H^2({\cal Q}\cup {\cal D}_{i,j}^c, \mathbb{Q})$ 
spanned by the restrictions of $P^*\kappa_1$ and $\eta$.
\end{proposition}  
\begin{proof}
  By Proposition \ref{addboundarycomponent},
if differentials in 
${\cal D}_{i,j}^c$ have a single zero of order $k_i+k_j$, then 
${\cal Q}\cup {\cal D}_{i,j}^c$ is a smooth complex orbifold, and otherwise
${\cal Q}\cup {\cal D}_{i,j}^c$  has a normal
crossing singularity along ${\cal D}_{i,j}^c$.
As a consequence, there is a well defined desingularization of
${\cal Q}\cup {\cal D}_{i,j}^c$ which is a smooth complex orbifold.
This desingularization contains $d\geq 1$ copies of
${\cal D}_{i,j}^c$ where $d$ is the multiplicity of
the zero of order $k_i+k_j$ in ${\cal D}_{i,j}^c$, and
${\cal Q}\cup {\cal D}_{i,j}^c$ is obtained from this
desingularization by identifying these $d$ copies
of ${\cal D}_{i,j}^c$.

As a consequence, for a smooth closed surface $B$ 
we can talk about a
smooth map $\phi:B\to {\cal Q}\cup {\cal D}_{i,j}^c$ 
which intersects ${\cal D}_{i,j}^c$ transversely.
Such a smooth map is the projection of a smooth map
of $B$ into the desingularization of ${\cal Q}\cup {\cal D}_{i,j}^c$
which intersects the preimage of the hypersurface
${\cal D}_{i,j}^c$ transversely in finitely many points.
Each of these intersection points then descends to
an intersection point of $\phi(B)$ with ${\cal D}_{i,j}^c$. 

  Consider the
  pull-back $\Pi^E:E=\phi^*P^*{\cal C}\to B$ of the universal
  curve to $B$.  This is 
  a smooth fiber bundle over $B$. 
Denote by $\Delta_{k_\ell}^E$ the pull-back of $\Delta_{k_\ell}$ to $E$. Since
$\phi(B)$ intersects ${\cal D}_{i,j}^c$ transversely in finitely many points, 
the restriction of the projection $\Pi_E:E\to B$ to $\Delta^E_{k_\ell}$ is a
branched multi-section, and $\Delta^E_{k_\ell}\subset E$ is a cycle which defines
a homology class $\delta_\ell=\delta_{k_\ell}\in H_2(E,\mathbb{Q})$.

Write $\delta=\sum_i k_i\delta_{k_i}$. 
In view of the fact that $c_1(\nu^*)\cup c_1(\nu^*)[E]=P^*\kappa_1(\phi_*[B])$, 
  Proposition \ref{identify} and naturality with respect to
  pull-back implies that 
  \begin{align}\label{deltadot}
  \delta\cdot \delta &=c_1(\nu^*\otimes \tau^*)(\delta)=
 (c_1(\nu^*)-c_1(\tau))\cup (c_1(\nu^*)-c_1(\tau))[E]\\
  &=
 P^*\kappa_1(\phi_*[B])-2(2g-2)\eta(\phi_*[B]).\notag \end{align} 
Recall that $c_1(\nu^*)\cup c_1(\tau)[E]=(2g-2)\eta(\phi_*[B])$. 

For the proof of the proposition, we analyze the evaluation of the 
the cohomology classes in $H^2(E,\mathbb{Q})$ which are Poincare
dual to the classes $\delta_\ell$. 
To this end we distinguish three cases.

{\sl Case 1:} $\ell\not=i,j$.

  By Lemma \ref{trivial}, the restriction of the 
  bundle $\nu^{\otimes (k_\ell+1)}$ to $\Delta_{k_\ell}\cap \Pi^{-1}{\cal Q}$ coincides with
  the restriction of the bundle $\tau^*$. 
  We claim that this holds true on $\Delta_{k_\ell}\cap \Pi^{-1}({\cal Q}\cup 
  {\cal D}_{i,j}^c)$. To this end 
  note that since $\ell\not=i,j$ by assumption, 
  the local computation carried out in Lemma \ref{trivial}
  is valid as well for preimages of points in ${\cal D}_{i,j}^c$.

By statement (2) of Lemma \ref{normalbundle}, 
in this case 
  $\Delta_{k_\ell}^E$ is a smooth 
  multisection of $E$, and its normal bundle can be identified
  with the restriction of the vertical tangent bundle $\nu$ of $E$.
  Since the line bundle $L_{k_\ell}$ on $P^*\overline{\cal Q}$ is
  dual to $\Delta_{k_\ell}$ in the sense of intersections, by naturality 
  of Chern classes under pull-back
  we conclude that 
  the class $\delta_\ell$ is Poincar\'e dual
  to the Chern class $\phi^*c_1(L_{k_\ell})$ of
  $\phi^*(L_{k_\ell})$.

 As the restriction of the projection $\Pi^E$ to $\Delta^E_{k_\ell}$ is 
  an unbranched covering of degree $d_\ell$ where $d_\ell$ is the multiplicity of the zero of 
  order $k_\ell$ in ${\cal Q}$, 
 using the transfer map in cohomology for this covering map $\Delta^E_{k_\ell}\to B$ 
   we obtain
  \begin{align}\label{kell}
  \kappa^{k_\ell}(\phi_*[B])&=\phi^*c_1(L_{k_\ell})(\delta_\ell)=
  \delta_\ell \cdot \delta_\ell \notag\\
  &=c_1(\nu)(\delta_\ell)=\frac{-1}{k_\ell +1}\phi^*c_1(\tau)(\delta_\ell)=\frac{-d_\ell}{k_\ell +1}\eta(\phi_*[B]).
  \end{align}
    
  {\sl Case 2:} $i=j=\ell$.

 By construction, in this case we have 
  $\Delta^E_{k_q}\cap \Delta^E_{k_u}=\emptyset$ for $q\not=u$. Thus we 
  obtain from equation (\ref{deltadot}),  
  from Lemma \ref{dual} and from equation (\ref{kell})
  the equation
  \begin{equation}\label{ki2}
  (P^*\kappa_1-2(2g-2)\eta)(\phi_*[B])
   =\sum_{\ell\not=i}\frac{-k_\ell d_\ell}{k_\ell+1}
  \eta(\phi_*[B])+k_i\kappa^{k_i}(\phi_*[B]).\end{equation}
Solving for $\kappa^{k_i}(\phi_*[B])$ shows that 
\begin{equation}\label{ki3}
\kappa^{k_i}(\phi_*[B])=\frac{1}{k_i}\bigl(P^*\kappa_1(\phi_*[B])
+(\sum_{\ell\not=i}\frac{k_\ell d_\ell}{k_\ell+1}-2(2g-2))\eta(\phi_*[B])\bigr)\end{equation}
and hence 
the restriction of the class $\kappa^{k_i}$ to ${\cal Q}\cup
{\cal D}_{i,j}^c$ is contained in
the subgroup generated
by $P^*\kappa_1$ and $\eta$. 

To show that $[{\cal D}_{i,i}^c]$ also is contained in the subgroup generated by 
$P^*\kappa_1$ and $\eta$, recall from 
 Proposition \ref{boundarycomp2} that 
 the pull-back of $\Delta_{k_i}$ to $E$ is a branched multisection of 
$E$, with a single branch point in each fiber of $E$ over the points 
$x_u\in B$ with $\phi(x_u)\in {\cal D}_{i,i}^c$,   
  and each of these branch points is a zero of order 
  $2k_i$ of the differential $\phi(x_u)$.
 Thus  by Lemma \ref{normalbundle} and
 Lemma \ref{trivial}, we deduce as in the proof of
 Lemma \ref{intersection1} that  
  \begin{equation}\label{ki}
 \kappa^{k_i}(\phi[B])=c_1(L_{k_i})(\delta_i)=\delta_i\cdot \delta_i=
  c_1(\nu)(\delta_i)+b\end{equation}
  where $b$ is the 
  number of intersection points between $\phi(B)$ and ${\cal D}_{i,i}^c$,
  counted with sign and multiplicities. 

On the other hand, we have
\begin{equation}\label{kappa10}
P^*\kappa_1(\phi_*[B])=\sum_j k_jc_1(\nu^*)(\delta_j)+2(2g-2)\eta(\phi_*[B]) 
\end{equation}
and therefore as in Section \ref{signatureasintersection}, we conclude that
\begin{equation}\label{kappa11}
-c_1(\nu)(\delta_i)=c_1(\nu^*)(\delta_i)=
\frac{1}{k_i} \bigl(P^*\kappa_1\phi_*[B]-\sum_{u\not=i}c_1(\nu^*)(\delta_u)-
2(2g-2)\eta(\phi_*[B]) \bigr).
\end{equation}

Since by equation (\ref{kell}) in Case 1 above, for all $u\not=i$ the value
$c_1(\nu)(\delta_u)$ is a multiple of $\eta(\phi_*[B])$, we conclude from equations
(\ref{ki},\ref{kappa11}) and the fact that the class $\kappa^{k_i}$ is a linear combination of 
the restriction to $\overline{\cal Q}$ of the 
classes $P^*\kappa_1$ and $\eta$ that the same holds true for the class
$[{\cal D}_{i,i}^c]$.  This completes the proof of the proposition in the case $k_i=k_j$. 

{\sl Case 3:} $\ell=i\not=j$. 

    Let $x\in B$ be such that $\phi(x)\in {\cal D}_{i,j}^c$.
  By modifying $\phi$ with an isotopy, we may assume that for some complex structure on $B$
  (whose orientation may be opposite to the orientation of $B$ in the case that the intersection index
  of $\phi$ is negative),  the map
 $\phi$ is a holomorphic embedding near $x$, and that
 the intersection of $\phi(B)$ with 
 ${\cal D}_{i,j}^c$ is transverse at $\phi(x)$.  The pull-back 
  $\Delta_{k_i}^E$ of $\Delta_{k_i}$ to $E$ contains a point $p_0$ in the fiber
  of $E$ over $x$ which is a zero of the differential $\phi(x)\in {\cal D}_{i,j}^c$ of degree
  $k_i+k_j$. Furthermore, 
 it follows from the 
  discussion in the proof of Proposition \ref{boundarycomp2} that 
  there are holomorphic local coordinates $(z,y)$ 
  for $E$ near $p_0=(0,0)$ whose range is a polydisk 
  $\{(z,y)\in \mathbb{C}^2\mid \vert z\vert <\epsilon, \vert y\vert <\epsilon\}$,
and  with the following properties.
  \begin{enumerate}
  \item In these coordinates, the projection $\Pi^E:E\to B$ is the second factor projection $(z,y)\to y$.
  \item  On the disk $\{(z,y)\mid \vert z\vert <\epsilon\}$, the differential $\phi(y)=\phi(\Pi^E(z,y))$ 
  is the projectivization of the 
  differential $\Phi(y)=(z-y)^{k_i}(z+y)^{k_j}dz$.
    \end{enumerate}
  
Note that $(z,y)\to \Phi(y)(z)$ defines a section of the pull-back of the tautological bundle $\tau$ on 
${\cal P}$ over the domain of the coordinates $(z,y)$. Moreover,
in these coordinates, the locus $\Delta^E_{k_i}$ is the diagonal
$D=\{(y,y)\mid \vert y\vert < \epsilon\}$.  
The normal bundle of this diagonal is spanned by the restriction to $D$ of the holomorphic
vector field  
$\frac{\partial}{\partial z}-\frac{\partial}{\partial y}$, and the section
$\Psi(z,y)=\frac{1}{(z-y)^{k_i}}(\frac{\partial}{\partial z}-\frac{\partial}{\partial y})$ 
of the holomorphic tangent bundle of the polydisk
is meromorphic, with a pole of order $k_i$ along $D$. 

Now $(z,y)\to \Phi(y)(z)$ also can be viewed as a local holomorphic section of the vertical cotangent bundle. 
Pairing this section $\Phi$ with the meromorphic vector field $\Psi$ 
defines an isomorphism between the restriction of the line bundle 
$\tau$ to the punctured disk $D-p_0\subset \Delta_{k_i}$ 
and the $k_i+1$-th power of the conormal bundle of $D-p_0$. 
This isomorphism is the restriction of the  
isomorphism constructed in 
the proof of Lemma \ref{trivial}. Since for 
$(y,y)\in D$ we have $\Phi(\Psi(y,y))=(2y)^{k_j}$, the rotation number
of the image of this isomorphism with respect to a section of the vertical cotangent
bundle which extends across $p_0$ 
equals $k_j$.

Using this analysis for all intersection points of $\phi(B)$ with ${\cal D}_{i,j}^c$, 
we conclude that the restriction of the bundle $\tau$ to $\Delta_{k_i}^E$ can be identified
with the bundle $(\nu^*)^{\otimes (k_i+1)} \otimes \xi^{k_j}$ where $\xi$ is the bundle with divisor
the zeros of order $k_i+k_j$ for the differentials in $\phi(B)$. 

As a consequence, we have 
\begin{equation}\label{kj1}
\kappa^{k_i}(\phi_*[B])=\delta_{k_i}\cdot \delta_{k_i}=c_1(\nu)(\delta_{k_i})=\frac{1}{k_i+1}(-d_i\eta(\phi_*[B])+k_j b)\end{equation}
where $b$ is the number of intersections of $\phi_*(B)$ and ${\cal D}_{i,j}^c$, counted with
sign and multiplicity, and where $d_i$ is the multiplicity of the zero of order $k_i$.
Similarly, the same equation also holds true if we replace $k_i$ by $k_j$. 

As $\sum_\ell k_\ell\kappa^{k_\ell}(\phi_*[B])=(P^*\kappa_1-(2g-2)\eta)(\phi_*[B])$, from Case 1 above
we infer that 
\begin{equation}\label{rightside}
(k_i\kappa^{k_i}+k_j\kappa^{k_j})(\phi_*[B])=(P^*\kappa_1+(\sum_{\ell\not= i,j}\frac{k_\ell d_\ell}{k_\ell +1}-(2g-2))\eta)
(\phi_*[B])\end{equation} 
and hence $k_i\kappa^{k_i}+k_j\kappa^{k_j}$ is contained in the subgroup of 
$H^2(\overline{\cal Q},\mathbb{Q})$ generated by $P^*\kappa_1$ and $\eta$.

On the other hand, by equation (\ref{kj1}) we know that 
\begin{align}
& k_ik_j (\frac{1}{k_i+1}+\frac{1}{k_j+1})b\notag\\
& =
(k_i\kappa^{k_j}+k_j\kappa^{k_j})(\phi_*[B])+(\frac{k_id_i}{k_i+1}
+\frac{k_jd_j}{k_j+1})\eta(\phi_*[B]).\notag\end{align}

This yields that indeed, the intersection number $b$ with 
${\cal D}_{i,j}^c$ is a rational linear combination of  $P^*\kappa_1$ and $\eta$.
This completes the proof of the proposition.
\end{proof}

\begin{remark}\label{componentcomp} 
The proof of Proposition \ref{tauto} also shows the following.
Let ${\cal D}^1,{\cal D}^2$ be two boundary components of codimension one of 
a stratum ${\cal Q}$. Let us assume that ${\cal D}^1,{\cal D}^2$ are distinct components
of the same stratum of projective abelian differentials. The cohomology class defined
by ${\cal D}^i$ can be represented in the form $aP^*\kappa_1+b\eta$ for both $i=1,2$, that
is, it is the same linear combination of the classes $P^*\kappa_1$ and $\eta$. 
\end{remark}

\begin{remark}
The computations in the proof of Proposition \ref{tauto} can be used to establish
an explicit formula for the classes of the boundary divisors 
${\cal D}_{i,j}^c$ in $H^2({\cal Q}\cup {\cal D}_{i,j}^c,\mathbb{Q})$. As these formulas
are rather involved and we do not know any interesting application, we omit
this discussion.
\end{remark}

\section{A stratification of the spin moduli space}\label{stratification}

The goal of this section is to prove Theorem \ref{stratificationthm}. 
We begin with the following well known

\begin{lemma}\label{nocompletecurve}
Let ${\cal V}\subset {\cal M}_g$ 
be any subvariety. If $\kappa_1=0$ on ${\cal V}$ then 
${\cal V}$ does not contain any complete complex subvariety.
\end{lemma}
\begin{proof}
We evoke the following result of Wolpert \cite{W86}:
There exists a holomorphic line bundle $L$ on ${\cal M}_{g}$ with 
Chern class $\kappa_1$, and there is a Hermitian metric 
on $L$ with curvature form $\omega=\frac{1}{2\pi^2}\omega_{WP}$ where
$\omega_{WP}$ is 
the Weil Petersson K\"ahler form on ${\cal M}_{g}$. In particular, $\omega$ is 
positive. As a consequence, 
if $V$ is a compact complex variety of dimension $k\geq 1$ 
and if $\zeta:V\to {\cal M}_{g}$ is a
holomorphic map which does not factor through a map from a variety of smaller dimension, then 
\begin{equation}\label{kappa1}
\kappa_1^k(\zeta(V))=\int_V(\zeta^*\omega)^k>0.\end{equation}
This shows that if ${\cal V}\subset {\cal M}_g$ is a complex subvariety which contains
a complete complex subvariety, then 
$\kappa_1\not=0$ on ${\cal V}$.
\end{proof}

The following is the main result of \cite{G20}. Its proof is completely elementary.

\begin{theorem}[Gendron \cite{G20}]\label{gendron}
A stratum of abelian differentials does not contains a nontrivial complete
complex subvariety.
\end{theorem}

Let ${\cal M}_{g,{\rm odd}}$ be the finite orbifold cover of ${\cal M}_g$
which is the moduli space of curves with 
odd theta characteristic. By definition, this is the quotient of Teichm\"uller space
by the finite index subgroup of the mapping class group 
${\rm Mod}(S_g)$ which preserves an
\emph{odd spin structure} on the surface 
$S_g$ of genus $g$. Such an odd spin structure is defined
as a quadratic form on $H_2(S_g,\mathbb{Z}/2\mathbb{Z})$ with odd 
Arf invariant (see \cite{KtZ03} for more information). 
Each of the curves $X\in {\cal M}_{g,{\rm odd}}$ admits an odd 
\emph{theta characteristic}, which by definition is a holomorphic 
line bundle $L$ whose square equals the canonical bundle of $X$
and such that
$h^0(X,L)$ is odd. The square of a holomorphic section of $L$ is a holomorphic
one-form on $X$ with all zeros of even multiplicity. 

All bundles over ${\cal M}_g$ will be pulled back to 
${\cal M}_{g,{\rm odd}}$ and will be denoted by the same symbols. 
Let $\overline{\cal Q}$ be the closure in ${\cal P}$ of the stratum 
${\cal Q}=\mathbb{P}{\cal H}(2,\dots,2)^{\rm odd}$ 
of projective
abelian differentials with all zeros of order two and odd spin structure.
Then the restriction of the projection 
$P:{\cal P}\to {\cal M}_{g,{\rm odd}}$ to $\overline{\cal Q}$ is surjective.

Recall that $\overline{\cal Q}$ admits a stratification of depth $g-1$
into subspaces ${\cal Q}_j$ of codimension $j-1$.
Here ${\cal Q}_j$ is the union of all components of strata
in $\overline{\cal Q}$ of codimension $j-1$. In particular, we have
${\cal Q}_1={\cal Q}$
and ${\cal Q}_{g-1}$ is the union of those components of $\mathbb{P}{\cal H}(2g-2)$ with an
odd spin structure \cite{KtZ03}.

For $r\geq 2$ let
\[{\cal M}_{g,{\rm odd}}^r=\{(X,L)\in 
  {\cal M}_{g,{\rm odd}}\mid h^0(X,L)\geq r+1\}.\]

The first part of the following statement is Clifford's theorem (for $g\leq 4$), the second
and third parts are due to 
Teixidor i Bigas \cite{TiB87}, in particular Theorem 2.13 in that article.

\begin{theorem}[Clifford, Teixidor i Bigas \cite{TiB87}]\label{clifford}
\begin{enumerate}\item
For $g\leq 4$, the locus 
  ${\cal M}_{g,{\rm odd}}^r$ is empty for 
  all $r\geq 2$.
 \item  For $g\geq 5$ the locus ${\cal M}_{g,{\rm odd}}^2$ has pure codimension 3 in
  ${\cal M}_{g,{\rm odd}}$.
\item Any component of ${\cal M}_{g,{\rm odd}}^r$ has dimension at
  most $3g-2r-2$.
\end{enumerate}
\end{theorem}

 
  To avoid technical difficulties we occasionally pass
  to a finite orbifold cover $\hat {\cal M}$ of 
${\cal M}_{g,{\rm odd}}$ which is a complex manifold. Then strata of abelian differentials
over $\hat {\cal M}$ are complex manifolds as well. The properties we are interested
in do not change by this modification. By abuse of notion, we still work with 
${\cal M}_{g,{\rm odd}}$, adopting the convention that whenever we talk 
about smooth complex orbifolds, by which we mean the quotient of a 
smooth complex manifold by a finite group of biholomorphic automorphisms.

Recall from \cite{KtZ03} that a \emph{hyperelliptic component}
of (projective) abelian differentials consists of differentials
on hyperelliptic curves which are invariant under the
hyperelliptic involution.
There are two such components in each genus $g\geq 3$, the components
$\mathbb{P}{\cal H}(g-1,g-1)^{\rm hyp}$ and
$\mathbb{P}{\cal H}(2g-2)^{\rm hyp}$. The projection $P$ maps each of these
components onto the locus ${\rm Hyp}$ of hyperelliptic curves in ${\cal M}_g$.

By \cite{KtZ03}, for $g\geq 4$ the stratum 
$\mathbb{P}{\cal H}(2g-2)$ has three connected components. There is an odd
non-hyperelliptic component $\mathbb{P}{\cal H}(2g-2)^{\rm odd}$,
the hyperelliptic 
component $\mathbb{P}{\cal H}(2g-2)^{\rm hyp}$ 
and an even component $\mathbb{P}{\cal H}(2g-2)^{\rm even}$.
The parity of the hyperelliptic component $\mathbb{P}{\cal H}(2g-2)^{\rm hyp}$
is odd if and only if $g\equiv 1,2$ mod $4$.
In the case $g=3$, the even component coincides with the
hyperellipitic component. To keep notations uniform, we put $\mathbb{P}{\cal H}(4)^{\rm even}=
\emptyset$.
We have

\begin{lemma}\label{project}
The image of the projection 
$P:\mathbb{P}{\cal H}(2g-2)^{\rm odd}\cup
\mathbb{P}{\cal H}(2g-2)^{\rm even}\to {\cal M}_g$ is disjoint from the hyperelliptic locus. 
\end{lemma}
\begin{proof}
Let $q\in \mathbb{P}{\cal H}(2g-2)-\mathbb{P}{\cal H}(2g-2)^{\rm hyp}$ 
be a projective abelian differential with a single zero on a Riemann surface $X$ 
which is not contained in the hyperelliptic component of
$\mathbb{P}{\cal H}(2g-2)$.
The zero of the projective differential $q$ 
is a Weierstrass point on $X$, and $q$ is uniquely determined
by this Weierstrass point. 

If $X$ is a hyperelliptic surface,
then as Weierstrass points are fixed by the hyperelliptic 
involution, the projective differential $q$ is invariant under the hyperelliptic involution. 
But this implies that $q$ is contained in the hyperelliptic component of 
$\mathbb{P}{\cal H}(2g-2)$, a contradiction.
\end{proof}

\begin{example}\label{genus3}
 If $g=3$ 
 then the closure $\overline{\cal Q}$ of
  ${\cal Q}=\mathbb{P}{\cal H}(2,2)^{\rm odd}$ in ${\cal P}$ 
  consists precisely
  of squares of projective sections of an odd theta
  characteristic. By the first part of Theorem \ref{clifford},
  this implies that the restriction of the
  projection
$P:{\cal P}\to {\cal M}_{g,{\rm odd}}$ to 
$\overline{\cal Q}$ 
is a biholomorphism. Since the spin structure of
the hyperelliptic component of ${\cal H}(4)$ is even \cite{KtZ03},  
we have $\overline{\cal Q}=\mathbb{P}{\cal H}(2,2)^{\rm odd}\cup\mathbb{P}
{\cal H}(4)^{\rm odd}$.

As the restriction of the projection $P$ to $\overline{\cal Q}$ is
a biholomorphism, it induces an isomorphism in cohomology. 
Now $H^2({\cal M}_{3,{\rm odd}},\mathbb{Q})=\mathbb{Q}$ is generated
by $\kappa_1$ 
\cite{H83,RW14}. Since
$P\mathbb{P}{\cal H}(4)^{\rm odd}$ is a divisor in ${\cal M}_{3,{\rm odd}}$
and hence defines a  second cohomology 
class $\xi\in H^2({\cal M}_{3,{\rm odd}},\mathbb{Q})$,
this class then is a multiple
of $\kappa_1$. Now a line bundle defined by a divisor is trivial on the complement
of the divisor, the restriction of $\kappa_1$ to $P{\cal H}(2,2)^{\rm odd}$ vanishes.

On the other hand, by Lemma \ref{project}, the 
divisor $P\mathbb{P}{\cal H}(4)^{\rm odd}$ is disjoint from the
hyperelliptic locus
${\rm Hyp}=P\mathbb{P}{\cal H}(4)^{\rm hyp}$
which is a divisor in ${\cal M}_{3,{\rm odd}}$. This divisor also defines
a multiple of $\kappa_1$. In other words, the Chern class of the line
bundle defined by the divisor
${\rm Hyp}$ is a multiple of $\kappa_1$. As the line bundle dual to a divisor
is trivial on the complement of the divisor, we conclude that
the restriction of 
$\kappa_1$ to $P\mathbb{P}{\cal H}(4)^{\rm odd}$ vanishes.

As a consequence, ${\cal M}_{3,{\rm odd}}$ is stratified into two 
strata, namely the stratum ${\cal M}_{3,{\rm odd}}-P\mathbb{P}{\cal H}(4)^{\rm odd}$ and 
the stratum $P\mathbb{P}{\cal H}(4)^{\rm odd}$, and the restriction of $\kappa_1$ to each of these strata
vanishes. In particular, these strata  do not contain a complete subvariety. 
Together we obtain Theorem \ref{stratificationthm} in the case $g=3$.
The article \cite{FL08} contains a stronger result. 
\end{example}

\begin{example}\label{genus4}
For $g=4$,  
Clifford's theorem shows that the restriction of the projection
$P:{\cal P}\to {\cal M}_{4,{\rm odd}}$ to 
the closure $\overline{\cal Q}$ 
of ${\cal Q}=\mathbb{P}{\cal H}(2,2,2)^{\rm odd}$ 
is a biholomorphism. 
Since the spin structure of the hyperelliptic component of 
$\mathbb{P}{\cal H}(6)^{\rm hyp}$ is even, 
we have
$\overline{\cal Q}= \mathbb{P}{\cal H}(2,2,2)^{\rm odd}\cup 
\mathbb{P}{\cal H}(2,4)^{\rm odd}\cup \mathbb{P}{\cal H}(6)^{\rm odd}$.

Since $H^2({\cal M}_{4,{\rm odd}},\mathbb{Q})=\mathbb{Q}$, we know that 
$P\mathbb{P}{\cal H}(2,4)^{\rm odd}$ is dual to a multiple of $\kappa_1$. 
This also follows from Proposition \ref{tauto}. Namely, 
as $P\vert \overline{\cal Q}$ is a biholomorphism and the class 
of the boundary divisor $\overline{\mathbb{P}{\cal H}(2,4)^{\rm odd}}$ 
in $\overline{\cal Q}$ is a rational linear combination of the class $\eta$ and 
$P^*\kappa_1$,  
the class of the divisor $P\mathbb{P}{\cal H}(2,4)^{\rm odd}$ is a multiple
of $\kappa_1$. In particular, the class $\kappa_1$ vanishes on 
${\cal M}_{g,{\rm odd}}-P\mathbb{P}{\cal H}(2,4)^{\rm odd}$. 

Similarly, by Proposition \ref{tauto}, the class of the boundary divisor
 $\mathbb{P}{\cal H}(6)^{\rm odd}$ in 
 $\overline{\mathbb{P}{\cal H}(2,4)^{\rm odd}}$ is a rational linear combination of 
 the class $P^*\kappa_1$ and $\eta$. Thus as before, the class of the 
 divisor $P\mathbb{P}{\cal H}(6)^{\rm odd}\subset 
 \overline{P\mathbb{P}{\cal H}(2,4)^{\rm odd}}$ 
 is a rational multiple of $\kappa_1$,
 and the restriction of $\kappa_1$ to $P\mathbb{P}{\cal H}(2,4)^{\rm odd}$
 vanishes. 
 
 This discussion can not be used to show that the restriction of 
 $\kappa_1$ to $P\mathbb{P}{\cal H}(6)^{\rm odd}$
 vanishes as well. To this end we
 need a different argument as explained below. Assuming this result,
 we obtain a stratification of ${\cal M}_{4,{\rm odd}}$ into
 $3$ strata such that the restriction of $\kappa_1$ to each of these strata 
 vanishes. A stronger result is contained in \cite{FL08}. 
\end{example}

Our next goal is to show that the restriction of $\kappa_1$ to 
$P\mathbb{P}{\cal H}(2g-2)\subset {\cal M}_g$ vanishes.
This then implies that the restriction of $\kappa_1$
to $P\mathbb{P}{\cal H}(2g-2)^{\rm odd}\subset
{\cal M}_{g,{\rm odd}}$ vanishes.

This vanishing statement is certainly well known.
As we were not able to locate a precise statement in the
literature, we provide a proof
which also illustrates the use of the results in Section \ref{boundary}
for applications beyond the strict context of that section.

A zero of order $2g-2$ for an abelian differential
on a Riemann surface of genus $g$ is a Weierstrass point.
As a complex curve of genus $g$ has $(g-1)g(g+1)$ Weierstrass points counted with multiplicity,
this implies that the restriction
of the projection $P$ to $\mathbb{P}{\cal H}(2g-2)$ 
is a finite morphism onto its image.
Although by Theorem \ref{strata} the restrictions of 
$P^* \kappa_1$ and $\eta$ to $\mathbb{P}{\cal H}(2g-2)$
are positive multiples of each other,
this does not immediately imply that $\kappa_1=0$ on $P\mathbb{P}{\cal H}(2g-2)$
as $\mathbb{P}{\cal H}(2g-2)$
may be a twisted multisection over its projection, similar
to a section of a trivial surface bundle over a surface
obtained from a map of non-zero degree from the base onto the fiber
(see \cite{H12}).

Instead we directly apply the results of Section \ref{boundary} towards
our goal. Namely, define ${\cal O}=\mathbb{P}{\cal H}(1,2g-3)$.
By \cite{KtZ03}, the stratum is connected. Furthermore, 
it contains the entire stratum $\mathbb{P}{\cal H}(2g-2)$ in its boundary.
The dimension of ${\cal O}$ equals $2g$.





The following is due to Gendron \cite{G18}. 

\begin{lemma}\label{projectspecial}
$P\overline{\cal O}\subset {\cal M}_{g}$ is a complex variety of dimension $2g$.
\end{lemma}

\begin{corollary}\label{multiple}
The restriction of $\kappa_1$ to $P\mathbb{P}{\cal H}(2g-2)$ 
vanishes.
\end{corollary}
\begin{proof}  
By Lemma \ref{projectspecial},  
$P\overline{\cal O}$ is a complex variety of dimension $2g$.
By Lemma \ref{project}, it contains two 
disjoint subvarieties ${\cal V}_1,{\cal V}_2$ of codimension one. Here ${\cal V}_1$ is 
the projection of the boundary components
$\mathbb{P}{\cal H}(2g-2)^{\rm odd}$, and ${\cal V}_2$
is the projection of $\mathbb{P}{\cal H}(2g-2)^{\rm hyp}$. 
Each of these varieties defines a dual cohomology class. 
By Proposition \ref{tauto} and Remark \ref{componentcomp}, 
each of the distinct boundary components of ${\cal O}$ defines the same linear
combination of the class $\kappa_1$ and $\eta$. 
As a consequence, each of the two components
${\cal V}_1,{\cal V}_2$ define the same multiple of $\kappa_1$ in 
$P\overline{\cal O}$. Since ${\cal V}_1$ and ${\cal V}_2$ are disjoint and 
being divisors, they
define a nontrivial cohomology class on $P\overline{\cal Q}$, and this class is 
a multiple of the restriction of $\kappa_1$. As ${\cal V}_1$ and ${\cal V}_2$ are disjoint,  
this implies as in the proof of Corollary \ref{multiple} 
that the restriction of $\kappa_1$ to ${\cal V}_1$ and ${\cal V}_2$ vanishes. 

Since in this argument, we may replace $\mathbb{P}{\cal H}(2g-2)^{\rm odd}$
by $\mathbb{P}{\cal H}(2g-2)^{\rm even}$, this implies the corollary. 
\end{proof}

Example \ref{genus4} and Corollary \ref{multiple} prove Theorem \ref{stratificationthm} for
$g=4$. Thus for the remainder of this article, we restrict to the case $g\geq 5$.

For the formulation of the following lemma, note that as strata are smooth complex 
suborbifolds
of ${\cal P}$, the intersection of the closure in ${\cal P}$ of a stratum 
with a fiber of 
$P:{\cal P}\to {\cal M}_{g,{\rm odd}}$ is a compact complex variety. 
We use Theorem \ref{strata} to show the following well known fact.

\begin{lemma}\label{completecurve}
Let $X\in {\cal M}_{g,{\rm odd}}^2$ and $j\geq 0$ be such that 
${\rm dim}({\cal Q}_j\cap P^{-1}(X))>0$. Then $X\in P(\cup_{\ell\geq j+1}{\cal Q}_{\ell})$. 
\end{lemma}
\begin{proof}
As both the fiber of ${\cal P}\to {\cal M}_{g,{\rm odd}}$ and the union of 
strata ${\cal Q}_j$ are smooth complex suborbifolds of ${\cal P}$, 
for each $X\in {\cal M}_{g,{\rm odd}}$ 
the intersection
${\cal Q}_j\cap P^{-1}(X)$ is a complex (possibly singular) variety. 

Let us assume that
this variety has a component $Y$ of positive dimension $k\geq 1$.
As $P^{-1}(X)$ is compact and $\cup_{\ell\geq j}{\cal Q}_\ell\subset 
{\cal P}$ is closed, either $X\in P(\cup_{\ell\geq j+1}{\cal Q}_{\ell})$ 
or $Y\subset {\cal Q}_j$ is compact.

In the second case, 
$Y$ is a complex subvariety
of $P^{-1}(X)\cap {\cal Q}_j$, and as $P^{-1}(X)$ is a complex projective space,
we conclude as in the proof of Lemma \ref{nocompletecurve} that
$\eta^k(Y)>0$. Namely, $\eta$ is the Chern class of the tautological bundle
over ${\cal P}$ whose restriction to a fiber is positive. 
On the other hand, $Y\subset P^{-1}(X)$ implies that
$P^*\kappa_1^k(Y)=0$. But this contradicts the fact that by Proposition \ref{degree2}, 
the restriction of  the class
$P^*\kappa_1$ to ${\cal Q}_j$ is a positive
multiple of $\eta$. The lemma is proven.
\end{proof}

The following statement illustrates the
main remaining step towards the proof of Theorem \ref{stratificationthm}.
Recall from Theorem \ref{clifford} the definition of the locus
${\cal M}_{g,{\rm odd}}^2$. Define 
\[{\cal Z}={\cal M}_{g,{\rm odd}}^2-
{\cal M}_{g,{\rm odd}}^4\] 
to be the locus of pairs $(X,L)$ with
$h^0(X,L)=3$.

\begin{proposition}\label{etaidentify}
For $g\geq 5$ the preimage $P^{-1}({\cal M}_{g,{\rm odd}}^2)\cap 
\overline{\cal Q}={\cal E}$ is 
a divisor in $\overline{\cal Q}$ which is dual
to the restriction of the class 
$\eta$ to $H^2(\overline{\cal Q},\mathbb{Q})$.
\end{proposition}
\begin{proof} By Theorem \ref{clifford}, 
for $g\geq 5$
 the locus ${\cal M}_{g,{\rm odd}}^2$ is
  of pure codimension $3$, and for $r\geq 4$ the dimension of the
  locus ${\cal M}_{g,{\rm odd}}^r\subset {\cal M}_{g,{\rm odd}}$ 
  is at most $3g-2r-2$. Thus 
  ${\cal M}_{g,{\rm odd}}^2$ is the closure 
  of 
  ${\cal Z}={\cal M}_{g,{\rm odd}}^2-{\cal M}_{g,{\rm odd}}^4$.

  The dimension of the preimage of ${\cal M}_{g,{\rm odd}}^2$
  in $\overline{\cal Q}$  equals $3g-4$ and hence this preimage is a divisor 
  ${\cal E}$ in
  the closure $\overline{\cal Q}$ of
  ${\cal Q}=\mathbb{P}{\cal H}(2,\dots,2)^{\rm odd}$.
  We have to show that ${\cal E}$
  is dual to $\eta$.

By counting dimensions, 
the codimension of the locus $P^{-1}{\cal M}_{g,{\rm odd}}^4\cap 
\overline{\cal Q}$ is at least three.
As a consequence, it suffices to show that ${\cal E}- P^{-1}({\cal M}_{g,{\rm odd}}^4)$
is dual to $\eta$ in the sense of intersections in the
complex orbifold $\overline{\cal Q}-P^{-1}({\cal M}_{g,{\rm odd}}^4)$.

If $(X,L)\in {\cal Z}={\cal M}_{g,{\rm odd}}^2-{\cal M}_{g,{\rm odd}}^4$ 
then $h^0(X,L)=3$. Thus as ${\cal Z}$
is a (non-closed)  
complex subvariety  of ${\cal M}_{g,{\rm odd}}$ of complex codimension 3,
the restriction of the projection
$P$ to $P^{-1}({\cal Z})\cap \overline{\cal Q}$ defines on 
$P^{-1}({\cal Z})\cap \overline{\cal Q}$ the structure of a
$\mathbb{C}P^2$-bundle over  ${\cal Z}$. For each point $x\in {\cal Z}$, 
the fiber of this bundle is a projective subplane of the fiber of 
${\cal P}$. In particular, the restriction of the fiberwise tautological 
line bundle for ${\cal P}$ to this projective plane coincides with the tautological
line bundle of this plane.  

Since the restriction of $P$ to
$\overline{\cal Q}-P^{-1}{\cal M}_{g,{\rm odd}}^2$ is a biholomorphism, 
$P^{-1}({\cal M}_{g,{\rm odd}}-{\cal M}_{g,{\rm odd}}^4)\cap \overline{\cal Q}$ 
is biholomorphic to 
the blow-up of the codimension three subvariety 
${\cal Z}$ in ${\cal M}_{g,{\rm odd}}-{\cal M}_{g,{\rm odd}}^4$
by uniqueness of blow-ups as explained on p.604 of \cite{GH78}. 
The normal bundle of the blow-up of ${\cal Z}$ 
is equivalent to the fiberwise tautological bundle 
over the blow-up fibers. By the discussion in the previous paragraph, this 
bundle is just the restriction of the line bundle $\tau$
over ${\cal P}$. By naturality
of Chern classes under pull-back by inclusions, the restriction to
$P^{-1}({\cal Z})$ of the Chern class
of this normal bundle equals the restriction of $\eta$.

Let $\pi$ be the restriction of the projection $P$ to
$\overline{\cal Q}$. 
The rational cohomology of the blow-up of ${\cal M}_{g,{\rm odd}}-
{\cal M}_{g,{\rm odd}}^4$ along ${\cal Z}$
equals
\[\pi^*H^*({\cal M}_{g,{\rm odd}}-{\cal M}_{g,{\rm odd}}^4,\mathbb{Q})
  \oplus H^*(\pi^{-1}({\cal Z}),\mathbb{Q})/\pi^*H^*({\cal Z},\mathbb{Q})\]
(see p.605 of \cite{GH78}), and the
cohomology $H^*(\pi^{-1}({\cal Z}),\mathbb{Q})$ is the cohomology of
a $\mathbb{C}P^2$-bundle over ${\cal Z}$ whose cohomology ring
is a quotient of $H^*({\cal Z},\mathbb{Q})[\hat \eta]$
where $\hat \eta$ is the restriction of $\eta$ (p.406 of \cite{GH78}).
Since $H^2({\cal M}_{g,{\rm odd}},\mathbb{Q})$
is spanned by $\kappa_1$, this yields that the class
defined by the divisor $\pi^{-1}({\cal Z})$ is of the form
$\eta\vert \overline{\cal Q}$.
\end{proof}


Define
\[{\cal Y}=P(\overline{\cal Q}-{\cal Q})\subset {\cal M}_{g,{\rm odd}}.\] 
For reasons of dimension,
${\cal Y}$ is a divisor in ${\cal M}_{g,{\rm odd}}$ which contains 
${\cal Z}$ as a subvariety of codimension two
by Lemma \ref{completecurve}.
Thus by Proposition \ref{etaidentify},  
${\cal W}=P^{-1}({\cal Y})\cap \overline{\cal Q}$ is
a divisor in $\overline{\cal Q}$ containing the closure 
${\cal E}$ of $P^{-1}({\cal M}_{g,{\rm odd}}^2)\cap \overline{\cal Q}$ as an irreducible 
component.

As $H^2({\cal M}_{g,{\rm odd}},\mathbb{Q})$
is generated by the class $\kappa_1$, 
the divisor  ${\cal Y}$ is dual to a
line bundle $\xi$ whose Chern class $c_1(\xi)$ is a multiple of $\kappa_1$.
The line bundle $\xi$ is trivial on
${\cal M}_{g,{\rm odd}}-{\cal Y}$, and its pull-back to
$\overline{\cal Q}$ is trivial on
$\overline{\cal Q}-P^{-1}{\cal Y}
\subset {\cal Q}$. By naturality of the duality between
divisors and line bundles
under birational maps, we conclude that
$P^{-1}{\cal Y}$ is the defining divisor for
the pull-back of $\xi$ to $\overline{\cal Q}$.
In other words, $P^*\kappa_1$ vanishes on
$\overline{\cal Q}-P^{-1}{\cal Y}$.

Recall that $P^{-1}{\cal Y}={\cal E}\cup (\overline{\cal Q}-{\cal Q})$
is reducible, and the irreducible component ${\cal E}$ dual to
$\eta$  
may intersect ${\cal Q}$ non-trivially. As a consequence, the
cohomology class in $H^2(\overline{\cal Q},\mathbb{Q})$ defined by 
the irreducible component $\overline{\cal Q}-{\cal Q}$
(which is the closure of the connected stratum
$\mathbb{P}{\cal H}(2,\dots,2,4)^{\rm odd}$) 
is a rational linear combination of the classes
$\eta$ and $P^*\kappa_1$ as was shown in Proposition \ref{tauto}.

Recall that we denoted by ${\cal Q}_j$ the union of the components of strata
in $\overline{\cal Q}$ of codimension $j-1$. We are now ready to 
complete the main step in the proof of Theorem \ref{stratificationthm} from the
introduction. Note
that by a result of Diaz \cite{D84}, the maximal dimension of a complete
subvariety of ${\cal M}_g$ and hence of ${\cal M}_{g,{\rm odd}}$ is not bigger than $g-2$.
We do not have information on a sharp bound.

\begin{proposition}\label{almostaffine}
For $k\leq g-1$ define ${\cal D}_k=P({\cal Q}_{k})-\overline{P({\cal Q}_{k+1})}$; then 
for all $k$, the restriction of $\kappa_1$ to the locus
${\cal D}_k$ vanishes. As a consequence, ${\cal D}_k$ 
does not contain a complete variety of positive dimension, and 
${\cal M}_{g,{\rm odd}}-\cup_{j\geq k+1}{\cal D}_j$ does not contain a complete variety of 
dimension at least $k$.
\end{proposition}
\begin{proof} By Corollary \ref{multiple}, it suffices to show the proposition in the case
$k\leq g-2$. Thus 
let $k\leq g-2$ and let 
${\cal A}$ be a component of ${\cal Q}_k$. This is a component of a stratum 
of abelian differentials with all zeros a multiple of 2 and odd spin structure. 
We have to show that 
the restriction of $\kappa_1$ to 
$P{\cal A}-\overline{P{\cal Q}_{k+1}}$ is trivial.
By Lemma \ref{nocompletecurve}, this then implies that 
$P{\cal Q}_k-\overline{P{\cal Q}_{k+1}}$ does not contain
a complete variety of positive dimension.

As the projection $P$ is closed, $P\overline {\cal A}$ 
 is a closed subvariety of ${\cal M}_{g,{\rm odd}}$.
Define ${\cal R}$ to be the closure of the 
set $\{z\in P\overline{\cal A}\mid {\rm dim}(P^{-1}(z)\cap \overline{\cal A})>0\}$.
Then ${\cal R}$ is a closed subvariety of $P\overline{\cal A}$ which is contained in 
$\overline{P{\cal Q}_{k+1}}$ by Lemma \ref{completecurve}. 
Since the restriction of $P$ to each component of a stratum in $\overline{\cal Q}$ 
is generically finite-to-one \cite{G18}, 
its codimension in $P\overline{\cal A}$ is at least one (in fact, ${\cal R}$  may
be empty). Furthermore, the preimage $\hat {\cal R}$ of ${\cal R}$ 
in $\overline{\cal A}$ is of codimension at least one as well. 

By naturality of pull-backs under birational maps, 
the pull-back by $P\vert \overline{\cal A}$ of the cohomology class dual to the divisor 
$P(\overline{\cal A}\cap \overline{\cal Q}_{k+1})$ 
is the cohomology class dual to $\hat {\cal R}\cup (\overline{\cal A}\cap 
\overline{\cal Q}_{k+1})$. 

Consider first the case that the codimension of $\hat {\cal R}$ is at least two.
Then 
this class coincides with the class 
defined by the divisor $\overline{\cal A}\cap \overline{\cal Q}_{k+1}$ in $\overline{\cal A}$.
By Proposition \ref{tauto}, this class is a linear combination of 
$P^*\kappa_1$ and $\eta$. By naturality under pull-back, we conclude that
$P(\overline{\cal A}\cap \overline{\cal Q}_{k+1})$ defines a multiple of the
restriction of $\kappa_1$. As a consequence, the restriction of 
$\kappa_1$ to 
$P\overline{\cal A}-P\overline{\cal A}\cap \overline{\cal Q}_{k+1}$
vanishes as claimed in the proposition.

If the codimension of $\hat {\cal R}$ equals one, then 
$P^{-1}(P(\overline{\cal A}\cap  \overline{\cal Q}_{k+1}))=
\hat {\cal R}\cup (\overline{\cal A}\cap \overline{\cal Q}_{k+1})$ is 
reducible, and it defines the pull-back of the class of 
$P(\overline{\cal A}\cap \overline{\cal Q}_{k+1})\subset P\overline{\cal A}$. 
By Proposition \ref{tauto}, the 
class of $\overline{\cal A}\cap \overline{\cal Q}_{k+1}$ is a rational linear combination of 
$P^*\kappa_1$ and $\eta$. On the other hand, we know that 
the class of the divisor $\hat {\cal R}$ is not the pull-back of a second cohomology
class on $P\overline{\cal A}$. 
As in the proof 
of Proposition \ref{etaidentify}, in this case we deduce that it defines a multiple of the 
fiber class $\eta$, and the divisor
$P(\overline{\cal A}\cap \overline{\cal Q}_{k+1})$ in $P\overline{\cal A}$  
defines a multiple of $\kappa_1$.
Together with 
Corollary \ref{multiple}, 
this completes the first part of the proposition.

To show the second part of the proposition, let $V\subset {\cal M}_{g,{\rm odd}}$ be a complete
variety of dimension $k\geq 1$ and assume that $V\subset {\cal M}_{g,{\rm odd}}-
{\cal D}_k$. As ${\cal M}_{g,{\rm odd}}-{\cal D}_2$ does not contain a complete variety,
the variety $V$ has to intersect ${\cal D}_2$ nontrivially. Since 
${\cal D}_2\subset {\cal M}_{g,{\rm odd}}$ is a closed subvariety of codimension one, 
this intersection is a complete variety $V_2$ whose dimension is at least $k-1$. 

Repeat this reasoning with $V_2\subset {\cal D}_2$ and the subvariety ${\cal D}_3$. 
In finitely many such steps we conclude that if 
$V\subset {\cal M}_{g,{\rm odd}}-{\cal D}_{k+1}$ 
has dimension $k$, then $V\cap {\cal D}_{k}$ is a complete variety of dimension at least 
one which is disjoint from ${\cal D}_{k+1}$. By the above, this is impossible. This completes
the proof of the proposition.
\end{proof}

\bigskip
\noindent
MATHEMATISCHES INSTITUT DER UNIVERSIT\"AT BONN\\
ENDENICHER ALLEE 60\\ 
53115 BONN, GERMANY

\bigskip
\noindent 
e-mail: ursula@math.uni-bonn.de
\end{document}